\documentclass[12pt]{amsart}
\usepackage{verbatim,amscd,amssymb}
\usepackage[color, notref, notcite]{showkeys}     
\definecolor{refkey}{gray}{.5}   
\definecolor{labelkey}{gray}{.5} 
\usepackage[active]{srcltx}
\usepackage{graphicx}

\newtheorem{thm}{Theorem}[section]
\newtheorem{prop}[thm]{Proposition}
\newtheorem{lem}[thm]{Lemma}

\newtheorem{cor}[thm]{Corollary}

\newtheorem*{thmS}{Straightening Theorem}

\theoremstyle{definition}
\newtheorem{dfn}[thm]{Definition}

\def\0{\emptyset}
  \def\Cc{\mathcal{C}}
 \def\Gc{\mathcal{G}}
\def\Fc{\mathcal{F}}   \def\Mc{\mathcal{M}}
\def\Pc{\mathcal{P}} \def\Rc{\mathcal{R}}  
\def\Uc{\mathcal{U}} \def\Vc{\mathcal{V}}  \def\Wc{\mathcal{W}}
   
  \def\Li{\mathcal{LI}} \def\chai{{\mathcal{CH}}}

\def\Qf{\mathfrak{Q}}
\def\Uf{\mathfrak{U}} \def\Vf{\mathfrak{V}}

\def\Z{\mathbb{Z}}

\def\C{\mathbb{C}}

\def\R{\mathbb{R}}

\renewcommand\emptyset{\varnothing}
\newcommand{\sm}{\setminus}

\def\ba{\mathbf{a}}
\def\bx{\mathbf{x}}
\def\bA{\mathbf{A}}
\def\bX{\mathbf{X}}

\def\eps{\varepsilon}
\def\ol{\overline}

\def\si{\sigma}  \def\ta{\theta}  
\def\al{\alpha}  \def\be{\beta}  \def\la{\lambda} \def\ga{\gamma}
\def\om{\omega}

\def\vp{\varphi}
\def\le{\leqslant}
\def\ge{\geqslant}
\def\wt{\widetilde}
\def\wh{\widehat}

\def\en{\mathrm{end}}
\def\thr{\mathrm{thr}}
\def\imr{\mathrm{ImR}}
\def\inn{\mathrm{Int}}

\def\cuc{\mathcal{CU}}
\def\ch{\mathrm{CH}}
\def\uc{\mathbb{S}^1}
\def\bd{\mathrm{Bd}}
\def\disk{\mathbb{D}}
\def\<{\langle}
\def\>{\rangle}
\def\phd{\mathrm{PHD}}
\def\cdisk{\overline{\mathbb{D}}}
\def\thu{\mathrm{Th}}

\def\cont{\mathrm{Shad}}

\begin{document}
\date{February 20, 2021}

\title[Immediate renormalization]
{Immediate renormalization of complex polynomials}

\author[A.~Blokh]{Alexander~Blokh}

\author[L.~Oversteegen]{Lex Oversteegen}

\author[V.~Timorin]{Vladlen~Timorin}

\thanks{The study has been funded by the Russian Academic Excellence Project `5-100'.}

\address[Alexander~Blokh and Lex~Oversteegen]
{Department of Mathematics\\ University of Alabama at Birmingham\\
Birmingham, AL 35294-1170}

\address[Vladlen~Timorin]
{Faculty of Mathematics\\
HSE University\\
6 Usacheva St., 119048 Moscow, Russia
}

\address[Vladlen~Timorin]
{Independent University of Moscow\\
Bolshoy Vlasyevskiy Pereulok 11, 119002 Moscow, Russia}

\email[Alexander~Blokh]{ablokh@math.uab.edu}
\email[Lex~Oversteegen]{overstee@math.uab.edu}
\email[Vladlen~Timorin]{vtimorin@hse.ru}




\begin{abstract}
A cubic polynomial $P$ with a non-repelling fixed point $b$
is said to be \emph{immediately renormalizable} if there exists
a (connected) quadratic-like invariant filled Julia set $K^*$ such that
$b\in K^*$. In that case exactly one critical point of $P$ does not
belong to $K^*$. We show that if, in addition, the Julia set of $P$ has no 
(pre)periodic cutpoints then this critical point is recurrent.  

\end{abstract}

\maketitle

\section{Introduction}

In the introduction we assume knowledge of basics of complex dynamics.

We study polynomials $P$ with connected Julia sets $J(P)$. An (external)
ray with a rational argument always lands at a point that is eventually
mapped to a repelling or parabolic periodic point. If two
external rays like that land at a point $x\in J(P)$, then
such rays are said to form a \emph{rational cut (at $x$)}. The family
of all rational cuts of a polynomial $P$ may be empty (then one says
that the \emph{rational lamination} of $P$ is empty); if it is
non-empty it provides a combinatorial tool allowing one to study
properties of $P$ even in the presence of such complicated irrational
phenomena as Cremer or Siegel periodic points.

Consider quadratic polynomials with connected Julia set. It is known
that any quadratic polynomials $Q\notin \ol{\phd_2}$ has
rational cuts ($\phd_2$ is the \emph{Principal Hyperbolic Domain} of the
Mandelbrot set). Thus, any 
arc from $\ol{\phd_2}$ towards the rest of the Mandelbrot set consists of polynomials with
rational cuts.

The purpose of this paper is to investigate a similar phenomenon in the
cubic case. Then there is a ``gray area'' $\Gc$ in-between the set of
cubic polynomials with rational cuts, and the set $\ol{\phd_3}$, the \emph{Principal
Hyperbolic Domain} of the cubic connectedness locus. We
conjecture that $\Gc$ is empty. Thus, the set $\Gc$ is the true object
of our study even though, paradoxically, in the end we want to
establish that it is empty. As a step in this direction we
prove that a polynomial from $\Gc$ must have specific
properties. Our approach is not unusual: according to the nature of the
contrapositive argument one studies a phenomenon in great detail only
to discover that an elaborate list of its properties leads to
contradictions thus disproving its existence.

\begin{dfn}[\cite{DH-pl}]\label{d:ql1}
A \emph{polynomial-like} map is a proper holomorphic map $f: U\to f(U)$
of degree $k>1$, where $U$, $f(U)\subset\C$ are open Jordan disks and
$\ol{U}\subset f(U)$. The \emph{filled Julia set} $K(f)$ of $f$ is the
set of points in $U$ that never leave $U$ under iteration of $f$. The
\emph{Julia set} $J(f)$ of $f$ is the boundary of $K(f)$. Call $U$ a
\emph{basic neighborhood} of $K(f)$ and assume that if $f$ is given,
then its basic neighborhood is fixed.
If $k=2$, then the corresponding polynomial-like maps are said to be
\emph{quadratic-like}.
\end{dfn}

We can now state our main result.

\begin{thm} \label{t:recur}
Let $f$ be a cubic polynomi\-al with empty rational lamination that has
a quadratic-like restriction with a connected quadratic-like filled Julia set $K^*(f)=K^*$. Then the critical point of
$f$ that does not belong to $K^*$ is recurrent.
\end{thm}

In the situation of Theorem \ref{t:recur} we will always denote a connected quadratic-like filled Julia set
by $K^*$; also, we will fix its neighborhood $U^*$ on which $f$ is quadratic-like and denote $f|_{U^*}$ by $f^*$.

\section{Preliminaries: polynomial-like maps and cubic polynomials with empty rational lamination}\label{s:prelim}

By \emph{classes} of polynomials, we mean affine conjugacy classes. For
a polynomial $f$, let $[f]$ be its class, let $K(f)$ be its filled Julia set,
and let $J(f)$ be its Julia set. The \emph{connectedness locus $\Mc_d$ of degree $d$} is the
set of classes of degree $d$ polynomials whose critical points \emph{do
not escape} (i.e., have bounded orbits). Equivalently, $\Mc_d$ is the
set of classes of degree $d$ polynomials $f$ whose Julia set $J(f)$ is
connected. The connectedness locus $\Mc_2$ of degree $2$ is called the
\emph{Mandelbrot set}; the connectedness locus $\Mc_3$ of degree $3$ is
called the \emph{cubic connectedness locus}.  The \emph{principal
hyperbolic domain} $\phd_3$ of $\Mc_3$ is defined as the set of classes
of hyperbolic cubic polynomials whose Julia sets are Jordan curves.
Equivalently, $[f]\in\phd_3$ if both critical points of $f$ are
in the immediate basin attraction of the same (super-)attracting fixed point.
A polynomial is \emph{hyperbolic} if the orbits of all critical points converge
to (super-)attracting cycles.

\subsection{Polynomial-like maps}\label{ss:pl}

\begin{dfn}[\cite{DH-pl}]\label{d:ql2}
Two polynomial-like maps $f:U\to f(U)$ and $g:V\to g(V)$ of degree $k$ are said to be
\emph{hybrid equivalent} if there is a quasi-conformal map $\vp$ from a
neighborhood of $K(f)$ to a neighborhood of $K(g)$ conjugating $f$ to
$g$ in the sense that $g\circ\vp=\vp\circ f$ wherever both sides are
defined and such that $\ol\partial\vp=0$ almost everywhere on $K(f)$.
\end{dfn}

The terminology is explained by the following fundamental result.

\begin{thmS}[\cite{DH-pl}]
Let $f:U\to f(U)$ be a polynomial-like map. Then $f$ is hybrid
equivalent to a polynomial $P$ of the same degree. Moreover, if $K(f)$
is connected, then $P$ is unique up to $($a global$)$ conjugation by an
affine map.
\end{thmS}

We will need the next definition.

\begin{dfn}\label{d:plrays}
Let $f$ be a polynomial, and for some Jordan disk $U^*$ the map
$f^*=f|_{U^*}$ be polynomial-like. Let $g$ be a monic polynomial hybrid
equivalent to $f^*$. Then the corresponding filled Julia set
$K(f^*)=K^*$ of $f^*$ is called the \emph{polynomial-like filled invariant
Julia set}.
The curves in $\C\sm K(f^*)=K^*$ corresponding (through the
hybrid equivalence) to dynamic rays of $g$ are called
\emph{polynomial-like rays} of $f^*$. If the degree of $f^*$ is two,
then we will talk about \emph{quadratic-like rays}. Denote
polynomial-like rays $R^*(\be)$, where $\be$ is the argument of the
external ray of $g$ corresponding to $R^*(\be)$. We will also call them
\emph{$K^*$-rays} to distinguish them from rays external to $K(f)=K$
called \emph{$K$-rays}.
\end{dfn}

Here $K^*$-rays are defined in a bounded neighborhood of
$K^*$ while $K$-rays are defined on the entire plane.
By Straightening Theorem
combined with well known results from complex dynamics
the composition $\psi^*: \C\sm K^*\to \C\sm \cdisk$ of the hybrid conjugacy for $K^*$ (see Straightening Theorem)
and the inverse Riemann map for $\C\sm K$ with real derivative at infinity
transfers the dynamics of $f$ to the plane on which,
loosely,
the role of $K^*$ is played by the unit circle while
the rest of the plane (i.e., the set $\C\sm K^*$) remains ``the same''.
Thus, the rest of $K$ (i.e.,  the set $K\sm K^*$) is
transferred to $\C\sm \cdisk$ and looks like a collection of pieces
``growing'' out of $\cdisk$. In terms of dynamics the map $P$ is transferred
by the map $\psi^*$ to a Jordan annulus $\wh {U^*}=\uc\cup \psi^*(U^*)$ to produce
the map $z^s: \wh {U^*}\to \uc\cup \psi^*(f(K^*))$ ($\uc\subset \C$ is the
unit circle centered at the origin) with the appropriate choice of $s$.


Evidently, if $f$ is a polynomial of degree $d$ and $T\subsetneqq J(f)$
is a proper polynomial-like invariant Julia set then the degree of
$f|_T$ is less than $d$. In particular, if $f$ is a \emph{cubic}
polynomial and $K^*\subset K(f)$ is a polynomial-like filled invariant
Julia set, then the polynomial-like map $f|_{K^*}$ is quadratic-like. The
following lemma is proven in \cite{bot16} (it is based upon Theorem
5.11 from McMullen's book \cite{mcm94}).

\begin{lem}[Lemma 6.1 \cite{bot16}]\label{l:7.2}
Let $f$ be a complex cubic polynomial with a non-repelling fixed point
$a$. Then the quadratic-like filled invariant Julia set $K^*$ with $a\in
K^*$ (if any) is unique.
\end{lem}

\subsection{Polynomials with empty rational
lamination}\label{ss:emp-lam}

As was said in the Introduction, we want to study cubic polynomials
$f\in \Mc_3$ without rational cuts (equivalently, with empty rational lamination).
This reduces the family of polynomials of interest to us.

\begin{lem}\label{l:gray-1}
Suppose that a cubic polynomial $f$ has empty rational lamination. Then
$f$ must have exactly one fixed non-repelling point and all other
periodic points of $f$ are repelling.
\end{lem}

\begin{proof}
By Theorem 7.5.2 \cite{bfmot12} if all fixed points of $f$ are
repelling then at one of them the combinatorial rotation number is not
$0$ and hence several rays must land, a contradiction. Also, if $f$ has
a fixed non-repelling point and a distinct periodic non-repelling
point, by Kiwi \cite{Ki} the rational lamination of $f$ is non-empty, a
contradiction.
\end{proof}

We will need the following corollary.

\begin{cor}\label{c:7.2}
A cubic polynomial $f$ with empty rational lamination contains, in its filled Julia set
$K(f)$, at most
one 
set $K^*$; this set
must contain a unique non-repelling fixed point of $f$.
\end{cor}

\begin{proof}
By Lemma \ref{l:gray-1}, the map $f$ has a unique non-repelling fixed point,
say, $a$, and all other periodic points of $f$ are repelling. Thus, if
$K^*$ does not
contain $a$, then all its periodic points are repelling. In particular,
by Theorem 7.5.2 \cite{bfmot12} the map $f^*$ 
has a fixed point $b$ such
that $K$-rays landing at $b$ rotate under $f^m$, 
a contradiction. Thus, 
$K^*$ contains $a$; by Lemma \ref{l:7.2} it is unique.
\end{proof}

\section{Preliminaries: full continua and their
decorations}\label{s:fucode}

In this section we consider certain ordered by inclusion pairs of
full continua on the plane (a compact set $X\subset\C$ is
\emph{full} if $\C\sm X$ is connected). This is a natural situation occurring in
complex dynamics, both when studying polynomials and their parameter
spaces. Indeed, let a cubic polynomial $f$ have a connected filled
Julia set $K(f)=K$. Suppose that $K^*$ and $U^*$ exist; 
Then the situation is
exactly like the one described above because $K*\subset K$. Another example is
when one takes the filled Main Cardioid of the Mandlebrot
set $\Mc_2$. It is easy to give other dynamical or parametric examples.

Let $X\subset Y$ be two full planar continua. We would like to
represent $Y$ as the union of $X$ and \emph{decorations}.

\begin{dfn}
\label{d:dec} Components of $Y\sm X$ are called \emph{decorations (of
$Y$ relative to $X$)}, or just \emph{decorations} (if $X$ and $Y$ are
fixed).
\end{dfn}

Decorations are connected but not closed; thus, decorations may behave
differently from what common intuition suggests. In Lemma \ref{l:triv}
we discuss topological properties of decorations. Given sets $A$ and
$B$, say that $A$ \emph{accumulates in} $B$ if $\ol{A}\sm A\subset B$.

\begin{lem}\label{l:triv}
Any decoration $D$ of $Y$ rel. $X$ accumulates in $X$. The set
$\ol{D}\sm D=\ol{D}\cap X$ is a continuum. The sets
$\ol{D}$ and $D\cup X=\ol{D}\cup X$
are full continua.
\end{lem}

\begin{proof}
Suppose, by way of contradiction, that there exists $x\in \ol{D}\sm
(D\cup X)$. Then we have $D\subset A=D\cup \{x\}\subset \ol{D}$ while
$A\cap X=\0$. Since $D$ is connected, and since $D\subset A\subset \ol{D}$,
then $A$ is connected too. Hence $D$ is not a
component of $K\sm X$, a contradiction.

The continuum $\ol{D}$ is full as otherwise its complementary domains
would be complementary to $Y$.

Now, by the first paragraph, $\ol{D}\sm D=\ol{D}\cap X$ is compact.
Suppose that $\ol{D}\cap X$ is disconnected. Then there exists a
bounded component $U$ of $\C\sm (X\cup D)$ that at least partially
accumulates to $X$ and partially to $D$. Since $Y$ is full, then
$U\subset Y$; hence $U$ is a subset of a decoration that accumulates
(partially) to points of $D$. By the first paragraph this implies that
$U\subset D$, a contradiction. Thus, $\ol{D}\cap X$ is connected; then
$\ol{D}\cap X$ is a full continuum as both $X$ and $\ol{D}$ are full.
\end{proof}

We will use the inverse Riemann map $\psi: \C\sm X\to \C\sm \cdisk$
with real derivative at infinity.
Loosely, one can say that under the map $\psi$ the
continuum $X$ is replaced by the closed unit disk $\cdisk$ while the
rest of the plane is conformally deformed. Thus, under $\psi$ the
decorations become subsets of $\C\sm \cdisk$.


\begin{cor}\label{l:triv2}
Let $D$ be a decoration of $Y$ rel. $X$. Then $\ol{\psi(D)}\sm D$ is a (perhaps,
degenerate) continuum (arc or unit circle) $I_D\subset \uc$.
\end{cor}

\begin{proof} Follows from Lemma \ref{l:triv}.
\end{proof}

Observe that the set $I_D$ can, indeed, coincide with the entire unit circle
(e.g., $D$ can spiral onto $\cdisk$). The arcs $I_D\ne \uc$ are
also possible as $\psi(D)$ may approach an arc $I_D$ by imitating the
behavior of the function $\sin(1/x)$ as $x\to 0^+$. Moreover, two
distinct decorations $D$ and $T$ may well have equal arcs $I_D$ and
$I_T$, or it may be so that, say, $I_D\subsetneqq I_T$, or $I_D$ and
$I_T$ can have a non-trivial intersection not coinciding with either
arc (all these examples can be constructed by varying the behavior of
components similar to the behavior function $\sin(1/x)$ as $x\to 0^+$).
However there are some cases in which one can guarantee that each
decoration has a degenerate arc $I_D$.

\medskip

\noindent\textbf{Ray Assumption on $X$ and $Y$.} \emph{Suppose that
there is a dense set $\mathcal A\subset \uc$ and a family of curves
$R_x$ landing in $\mathcal A$ and disjoint from $\psi(Y)$.}

\medskip

Suppose that Ray Assumption holds for $X$ and $Y$. Moreover, suppose that there is a
neighborhood $U$ of $Y$ and a homeomorphism $\varphi:U\to W\subset \C$. Then
Ray Assumption holds for $\varphi(X)\subset \varphi(Y)$ too.

\begin{lem}\label{l:triv3}
Suppose that Ray Assumption holds for $X\subset Y$. Then for every
decoration $D$ the arc $I_D$ is degenerate.
\end{lem}

\begin{proof}
Suppose that $I_D$ is a non-degenerate arc. Choose three points $x, y,
z\in \mathcal A\cap I_D$. It follows that $\psi(D)$ is contained
in one of the two disjoint open strips formed by the curves $R_x$,
$R_y$ and $R_z$. However then $\psi(D)$ can accumulate to only one
of the circle arcs formed by the points $x, y$ and $z$, a
contradiction.
\end{proof}

\begin{dfn}
Under Ray Assumption and in the above notation, the \emph{argument} of a
decoration $D$ is the angle $\al$ such that $I_D=\{\al\}$.
\end{dfn}

Since we study decorations in the complex dynamical setting,
making the ray assumption is not overly restrictive because, as we will
now see, it holds in important for us dynamical cases.

\begin{lem}\label{l:raya1}
Suppose that $K^*$ is a connected filled invariant polynomial-like Julia set contained
in a connected filled Julia set $K$ of a polynomial $P$. Then
$K^*\subset K$ satisfy the ray assumption.
\end{lem}

\begin{proof}
Choose a periodic repelling point $x\in K^*$ and a $K$-ray $R$ landing
at $x$. Under $\psi^*$ this ray becomes a curve $\psi^*(R)$
landing at the appropriate point of $\uc$, and these points are dense in $\uc$. 
The remark after we define Ray Assumption now shows, that
$K^*\subset K$ satisfy it.
\end{proof}

\section{Cubic parameter slices}\label{s:para-sli}

Let $\Fc$ be the space of polynomials
$$
f_{\lambda,b}(z)=\lambda z+b z^2+z^3,\quad \lambda\in \C,\quad b\in \C.
$$
An affine change of variables reduces any cubic polynomial to the form
$f_{\lambda,b}$. Clearly, $0$ is a fixed point for every polynomial in
$\Fc$. Define the \emph{$\la$-slice} $\Fc_\lambda$ of $\Fc$ as the
space of all polynomials $g\in\Fc$ with $g'(0)=\lambda$, i.e.
polynomials $f(z)=\lambda z+b z^2+z^3$ with fixed $\lambda\in \C$. We
write $\Pc_\la$ for the set of polynomials $f\in\Fc_\la$ for which
there are polynomials $g\in\Fc$ arbitrarily close to $f$ with
$|g'(0)|<1$ and the Julia set being a Jordan curve. Clearly, the class
$[f]$ of $f$ then belongs to $\ol\phd_3$. Also, denote by $\Fc_{nr}$
the space of polynomials $f_{\lambda,b}$ with $|\la|\le 1$ (``nr''
from ``\textbf{n}on-\textbf{r}epelling''). 
For a fixed $\la$ with $|\la|\le 1$ the \emph{$\la$-connectedness locus
$\Cc_\lambda$}, of the \emph{$\la$-slice} of the cubic connectedness
locus is defined  as the set of all $f\in\Fc_\la$ such that
$K(f)$ is connected. 
This is a
full continuum \cite{BrHu, Z}. We study sets
$\Cc_\la\subset \Fc_\la$ as we want to see to what extent
results concerning the quadratic Mandelbrot set hold for
$\Cc_\la$.

\subsection{\emph{Immediately renormalizable} polynomials vs the closure of
the principal hyperbolic component}\label{ss:imre-vs-hyp}

Let us describe what happens to quadratic-like invariant filled Julia
sets of $f\in \Fc_{nr}$ that contain $0$ when $f$ is slightly
perturbed (assuming such a set exists for a given $f$).
In the rest of the paper by the ``quadratic counterpart'' of $f^*$ we mean
a quadratic polynomial hybrid conjugate to $f^*$ by Straightening Theorem.

\begin{lem}\label{l:cnct} Let $f\in \Fc_{nr}$ be a polynomial, $K^*$
be a quadratic-like filled invariant Julia set containing $0$. Then
$K^*$ is connected. Every cubic polynomial $g\in \Fc_{nr}$ sufficiently
close to $f$ has a quadratic-like Julia set $B^*$ containing $0$; the
set $B^*$ here is also connected. Moreover, if $0$ is an attracting
fixed point for $g$ then $g$ has a quadratic-like Julia set which is a
Jordan curve; in particular, $g\notin \phd_3$.
\end{lem}

\begin{proof}
Since $0$ is non-repelling, then, by the Fatou-Shishikura inequality,
the critical point of the quadratic counterpart of $f^*$ cannot escape.
Hence $K^*$ is connected. Let $f^*:U^*\to V^*, f^*=f|_{U^*}$ be the
associated quadratic-like map. If $g$ is very close to $f$ then $0\in
U^*$ and, moreover, we can arrange for a new Jordan disk $W^*$ with $g(W^*)=V^*$.
By the above, the associated quadratic-like Julia
set $B^*$ is connected. Finally, if $f_i\to f$ are polynomials with $0$
as an  attracting fixed point then, by the above, for large $i$ the
polynomial $f_i$ has a quadratic-like filled Julia set coinciding with
the closure of the basing of immediate attraction of $0$. Therefore,
$[f_i]\notin \phd_3$ for large $i$ as desired.
\end{proof}

Call a cubic polynomial $f\in\Fc_{nr}$ \emph{immediately
renormalizable} if there are Jordan domains $U^*$ and $V^*$ such that
$0\in U^*$, and $f^*=f:U^*\to V^*$ is a quadratic-like map; denote by
$K^*$ the filled quadratic-like Julia set of $f^*$ (in the future we
\emph{always} use the notation $U^*,$ $V^*,$ $f^*$ and $K^*$ when talking about
immediately renormalizable maps). Denote the set of all \textbf{im}mediately
\textbf{r}enormalizable polynomials by $\imr$, and let $\imr_\la=\Fc_\la\cap
\imr$. Let $\Pc$ be the set of polynomials $f\in\Fc_{nr}$ with
the following property: there are polynomials $g\in\Fc$ arbitrarily
close to $f$ with $|g'(0)|<1$ such that $[g]\in\phd_3$. Then clearly
$[f]\in\ol\phd_3$ (observe, that there may be polynomials outside of
$\Pc$ whose classes are also in $\ol\phd_3$). Also, set
$\Pc_\la=\Pc\cap \Fc_\la$. Corollary \ref{c:cnct} follows from
Lemma \ref{l:cnct}.

\begin{cor}\label{c:cnct} If $f\in \imr$, then $K^*$ is connected.
The set $\imr$ is open in $\Fc_{nr}$. The set $\imr_\la$ is open in
$\Fc_\la$ for any $\la, |\la|\le 1$. The sets $\imr$ and $\Pc$ are
disjoint.
\end{cor}

We want to study the sets $\imr$ and $\Pc$; Corollary \ref{c:cnct} shows
that they are disjoint, so this investigation may be done in parallel.
In fact a lot is known about the sets $\imr_\la$ and $\Pc_\la$ in the
case when $|\la|<1$. Namely, P. Roesch proved the following theorem.

\begin{thm}[\cite{roe06}]\label{t:roesch}
$\Pc_\la$ is a Jordan disk for any $\la$ with $|\la|<1$.
\end{thm}

We want to combine this result with \cite{bopt16a} where sufficient
conditions on polynomials for being immediately renormalizable are
given.

\begin{thm}[\cite{bopt16a}]\label{t:when-imr}
If $f\in \Fc_\la, |\la|\le 1$ belongs to the unbounded complementary
domain of $\Pc_\la$ in $\Fc_\la$, then $f$ is immediately
renormalizable.
\end{thm}

Combining these theorems and Lemma \ref{l:cnct}, we get Corollary
\ref{c:<1}.

\begin{cor}\label{c:<1}
If $|\la|<1$ then $\imr_\la=\C\sm \Pc_\la$.
\end{cor}

We study the neutral case $|\la|=1$; to obtain more general results, as much
as possible we study neutral slices without using specifics of the number $\la$.

\begin{lem}\label{l:lambda<0}
If $|\la|<1$ then $[z^3+\la z]\in \phd_3$.
\end{lem}

\begin{proof}
We claim that $J(f)$ is a Jordan curve. Let $U$ be the basin
of immediate attraction of $0$ (which is an attracting fixed point of $f$). Since $f^n(-z)=-f^n(z)$ for every $n$ then $U$ is centrally
symmetric with respect to $0$. Since there exists a
critical point $c\in U$ and $f'(-z)=f'(z)$ for any $z$, then
$-c$ is critical. The central symmetry of $U$ with respect to $0$ now implies
that $-c\in U$. Since both critical points of $f$
belong to $U$, the claim follows.
\end{proof}

Theorem \ref{l:0in} develops Lemma \ref{l:lambda<0}.

\begin{thm}\label{l:0in}
For any $\la, |\la|<1$, we have $0\in \inn(\Pc_\la)$.
For any $\la, |\la|\le 1$, we have $0\in \Pc_\la$, and $\Pc_\la$ is a continuum.
\end{thm}

\begin{proof}
The first claim is proven in Lemma \ref{l:lambda<0}.
To prove the rest, observe that if $|\la|=1$ then $\Pc_\la=\limsup
\Pc_\tau$ where $\tau\to \la, |\tau|<1$. Since $0\in
\inn(\Pc_\tau)$ for all these numbers $\tau$, then $0\in
\Pc_\la$, and $\Pc_\la$, being the lim sup
of the continua $\Pc_\tau$ which all share a common point $0$, is also
a continuum as claimed.
\end{proof}

\subsection{The structure of the slice $\Fc_\la$}\label{ss:stru-fla}

Following \cite{bot16}, define the set $\cuc_\la, |\la|\le 1$ as the
set  of all polynomials $f\in \Fc_\la$ with connected Julia sets and
such that the following holds:

\begin{enumerate}

\item $f$ has no repelling periodic cutpoints in $J(f)$;

\item $f$ at most one non-repelling cycle not equal to $0$, and
all its points have multiplier $1$.

\end{enumerate}

The set $\cuc_\la$ is a centerpiece, literally and figuratively, of
the $\la$-slice $\Cc_\la$ of the cubic connectedness locus. A big role
in studying polynomials from $\Cc_\la$ is played by studying properties
of the quadratic polynomial $z^2+\la z$ whose fixed point $0$
has multiplier $\la$. Aiming at most general results, we consider the general case
of $\la$ with irrational argument without any extra-conditions. 
To state some theorems proven earlier 
we need new notions.
For a closed subset $A\subset \uc$ of at least 3 points,
call its convex hull $\ch(A)$ a \emph{gap}. Given a chord $\ell=\ol{ab}$ of the unit circle with
endpoints $a$ and $b$, set
$\si_3(\ell)=\ol{\si_3(a)\si_3(b)}$ (we abuse the notation and identify the angle-tripling
map $\si_3:\R/\Z$ with the map $z^3:\uc\to \uc$; similarly we treat the map $\si_2$). For a closed set $A\subset \uc$, call
each complementary arc of $A$ a \emph{hole} of $A$. Given a compactum $A\subset \C$ let the
\emph{topological hull} $\thu(A)$ be the complement to the unbounded complementary domain of $A$.

\subsubsection{Family of invariant quadratic gaps}\label{sss:family}

Let us discuss properties of \emph{quadratic}
$\si_3$-invariant gaps \cite{BOPT}.
For our purposes it suffices to consider gaps $G$ such that
$G\cap \uc$ has no isolated points.
``Invariant'' means that an edge of a gap $G$ maps to an edge of $G$, or to
a point in $G\cap \uc$;
``quadratic'' means that after collapsing holes 
of $G$
the map $\si_3|_{\bd(G)}$ induces a locally strictly monotone two-to-one map of the unit
circle to itself that preserves orientation and has no critical points.
For convenience, normalize the length of the circle so that it equals
$1$. Let $\Vf$ be a quadratic $\si_3$-invariant gap with no isolated
points. Then there is a unique arc $I_{\Vf}$ (called the \emph{major
hole} of $\Vf$) complementary to $\Vf\cap \uc$ whose length is greater
than or equal to $1/3$; the length of this arc is at most $1/2$. The
edge $M_{\Vf}$ of $\Vf$ connecting the endpoints of $I_{\Vf}$ is called the
\emph{major} of $\Vf$. If $M_{\Vf}$ is
critical then itself and $\Vf$ are said to be of \emph{regular critical
type}; if $M_{\Vf}$ is periodic then itself and $\Vf$ are said to be of
\emph{periodic type}. Collapsing edges of $\Vf$ to points, we construct
a monotone map $\tau:\Vf\to \uc$ that semiconjugates
$\si_3|_{\bd(\Vf)}$ and $\si_2:\uc\to \uc$.

The map $\tau$ is uniquely defined by the fact that it is monotone and
semiconjugates $\si_3|_{\bd(\Vf)}$ and $\si_2:\uc\to \uc$. Indeed, if
there had been another map $\tau'$ like that then there would have
existed a non-trivial orientation preserving homeomorphism of the
circle to itself conjugating $\si_2$ with itself. However it is easy to
see that the only such map is the identity (recall, that $\si_2$ has a unique fixed point).

The family of all invariant quadratic gaps can be parameterized (see
\cite{BOPT}). Namely, by Lemmas 3.22 and 3.23 of \cite{BOPT}, there
exists a Cantor set $Q\subset\uc$ such that if we
collapse every hole of $Q$ to a point, we obtain a topological circle
whose points are in one-to-one correspondence with \emph{all} quadratic
invariant gaps $\Uf$ such that $\Uf\cap \uc$ is a Cantor set. Moreover,
the following holds:

\begin{enumerate}

\item for each point $\ol\theta\in\uc$ of $Q$ that is not an
    endpoint of a hole of $Q$, the critical chord
    $\ol{(\theta+1/3)(\theta+2/3)}$ is the major of a quadratic
    invariant gap $\Uf$ such that $\Uf\cap \uc$ is a Cantor set;

\item for each hole $(\ol\theta_1,\ol\theta_2)$ of $Q$ the chord
    $\ol{(\theta_1 + 1/3)(\theta_2 + 2/3)}$ is the periodic major
    of a quadratic invariant gap $\Uf$ such that $\Uf\cap \uc$ is a
    Cantor set.

\end{enumerate}

Thus, to choose the \emph{tag} of a quadratic invariant gap $\Vf$ we
first take its major $M_\Vf$ and then choose the edge or vertex $\ell$
of $\Vf$ distinct from $M_\Vf$ but with the same image as $M_\Vf$.
Evidently, (1) $\ell$ is an edge of $\Vf$ if $M_\Vf$ is not critical
(and is, therefore, of periodic type), (2) $\ell$ is a vertex of $\Vf$
if $M_\Vf$ is critical (and is, therefore, of regular critical type).

The convex hull $\Qf$ of $Q$ in the plane is called the \emph{Principal
Quadratic Parameter Gap} (see Figure \ref{fig:Qf}). The set $Q$ plays a
somewhat similar role to that of the following set appearing in
quadratic dynamics: the set of arguments of all parameter rays (rays to
the Mandelbrot set) landing at points of the main cardioid. Holes of
$\Qf$ will play an important role. The \emph{period of a hole}
$(\ta_1,\ta_2)$ of $\Qf$ is defined as the period of $\ta_1+1/3$ under
the angle tripling map. This period also equals to the period of
$\ta_2+2/3$ under the angle tripling map. The only period 1 (invariant) holes of
$\Qf$ are $(1/6, 1/3)$ and $(2/3, 5/6)$; these will play a special
role.

\begin{figure}
  \includegraphics[height=6cm]{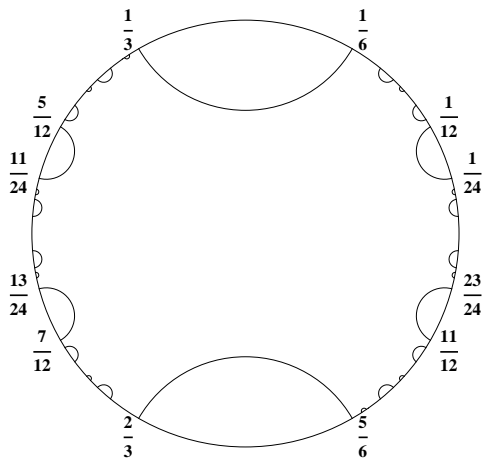}
  \caption{The Principal Quadratic Parameter Gap}
  \label{fig:Qf}
\end{figure}

\subsubsection{Properties of the $\la$-slice}\label{sss:la-slice}

For every polynomial $f\in\Fc_\la$ and every angle $\al\in\R/\Z$, we
will define the \emph{dynamic ray} $R_f(\al)$. Also, for every angle
$\theta$, in the parameter plane of $\Fc_\la$ we define the
\emph{parameter ray} $\Rc_\lambda(\theta)$. We use rays to show that
the picture in $\Fc_\la$ resembles the picture in the parameter plane
of quadratic polynomials.


\begin{thm}[Main Theorem of \cite{bot16}]\label{t:main-bot16}

Fix $\la$ with $|\la|\le 1$. The set $\cuc_\la$ is a full continuum.
The set $\Cc_\la$ is the union of $\cuc_\la$ and a countable family of
\emph{limbs} $\Li_H$ of $\Cc_\la$ parameterized by holes $H$ of $\Qf$.
The union is disjoint. 
For a hole $H=(\theta_1,\theta_2)$ of $\Qf$, the following holds.

\begin{enumerate}

\item The parameter rays $\Rc_\lambda(\theta_1)$ and
    $\Rc_\lambda(\theta_2)$ land at the same point $f_{root(H)}$.

\item Let $\Wc_\lambda(H)$ be the component of $\C\sm
    \ol{\Rc_\lambda(\theta_1)\cup \Rc_\lambda(\theta_2)}$ containing
    the parameter rays with arguments from $H$. Then, for every
    $f\in\Wc_\la(H)$, the dynamic rays $R_f(\theta_1+1/3)$,
    $R_f(\theta_2+2/3)$ land at the same point, either a periodic
    and repelling point for all $f\in\Wc_\la(H)$, or the point $0$ for all
    $f\in\Wc_\la(H)$. Moreover, $\Li_H=\Wc_\lambda(H)\cap \Cc_\la$.

\item The dynamic rays $R_{f_{root(H)}}(\theta_1+1/3)$,
    $R_{f_{root(H)}}(\theta_2+2/3)$ land at the same parabolic
    periodic point, and $f_{root(H)}$ belongs to $\cuc_\la$.
\end{enumerate}
\end{thm}

Figure \ref{fig:Fc1d3} shows the parameter slice $\Fc_{e^{2\pi i/3}}$ in which several
parameter rays and several wakes are shown.

\begin{figure}
  \includegraphics[height=8cm]{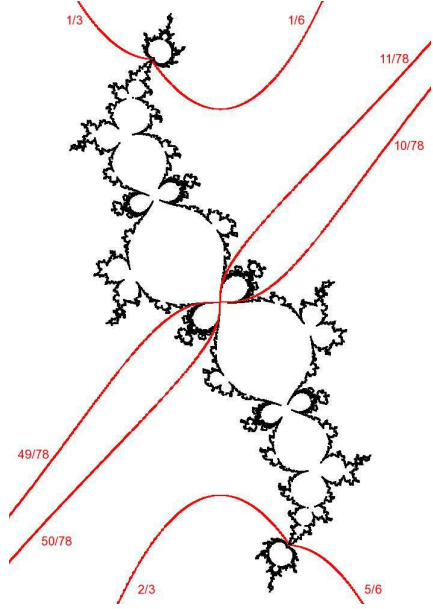}
  \caption{Parameter slice $\Fc_{e^{2\pi i/3}}$ with some parameter rays}
  \label{fig:Fc1d3}
\end{figure}

Given a compact set $A\subset \C$, let the \emph{topological hull of $A$} be the
unbounded complementary domain of $A$.

\begin{thm}[\cite{bot16}]\label{t:cubio}
We have that $\thu(\Pc_\la)\subset \cuc_\la$. The set $\cuc_\la$ is a
full continuum.
\end{thm}

In this paper we study the set $\cuc_\la\sm \thu(\Pc_\la)$.
By Theorems \ref{t:main-bot16} and \ref{t:cubio}, except
for vertices  $f_{root(H)}$ of parameter wakes $\Wc_\lambda(H)$, no
points of $\cuc_\la$ belong to those parameter wakes.
The only places at which points of $\cuc_\la\sm \thu(\Pc_\la)$ may be
located can be associated with parameter rays $\Rc_\la(\ta)$ where
$\ta\in \Qf$ is a parameter, associated with a regular critical
quadratic invariant gap $\Uf$ with regular critical major $M_{\Uf}$.

\section{Decorations}

The notation below will be used in what follows. Namely, we
assume that $f$ is an immediately renormalizable cubic polynomial; 
recall that by Lemma \ref{l:7.2} its filled Julia set $K^*$ is unique. 
Let $f:U^*\to V^*$ be a quadratic-like map where $V^*$ is very tight around $K^*$.
Let $\om_1$ be the critical point of $f$ belonging to $K^*$; let
$\om_2$ be the other critical point of $f$ (this notation will be used in what follows) . In order to indicate the
dependence on $f$, we may 
write $K^*(f)$, $\om_2(f)$, etc.
Observe that in this section we put \emph{no} restrictions on the
rational lamination of $f$. A \emph{pullback} of a
connected set $D$ under $f$ is defined as a connected component of
$f^{-1}(D)$.

\begin{lem}\label{l:pull}
A pullback of a connected set $A$ under $f$ maps onto $A$. In
particular, there exists a pullback $\wt K^*$ of $K^*$ disjoint from
$K^*$ and mapped by $f$ onto $K^*$ in the one-to-one fashion.
\end{lem}

\begin{proof} The first claim of the lemma is proven in Lemma 4.1 of \cite{lt90}.
Moreover, since $K^*$ is a quadratic-like Julia set, it has a pullback
$\wt K^*$ disjoint from itself. By the first claim, $\wt K^*$ is a
continuum that maps \emph{onto} $K^*$. Moreover, $f|_{\wt K^*}$ is
one-to-one. Indeed, otherwise there are points $x\ne y\in \wt K^*$ with
$f(x)=f(y)=z\in K^*$. It follows that $z$ has overall four preimages
(two in $K^*$ and two in $\wt K^*$), a contradiction.
\end{proof}

The notation $\wt K^*$ will be used from now on. Also, from now on
by decorations we mean those of $K$ rel. $K^*$.


\begin{dfn}
A decoration is said to be \emph{critical} if it contains $\wt K^*$.
Thus there is only one critical decoration denoted $D_c$.
All other decorations are said to be \emph{non-critical}.
\end{dfn}

Let $v_2=f(\om_2)$ be the critical value associated with the point $\om_2$.

\begin{lem}\label{l:v-nin-k} Neither $\om_2$ nor $v_2$ belong to $K^*$. 
\end{lem}

\begin{proof} We have $\om_2\notin K^*$ since $\om_1\in K^*$ and $K^*$ contains at most one critical point.
If $v_2\in
K^*$, then there are two preimages of $v_2$ in $K^*$ and two preimages
of $v_2$ outside of $K^*$ (both numbers take multiplicities into
account). This contradicts $f$ being three-to-one. 
\end{proof}

For $x\notin K^*$ let $D(x)$ be the decoration containing $x$; set
$D_v=D(v_2)$ and call it \emph{critical value decoration}. Initial
dynamical properties of decorations are listed in Theorem
\ref{t:decor}. 
Set $L$ to be $\{\om_2\}$ (if $\om_2\in J(f)$), or the closure of the Fatou domain of
$f$ containing $\om_2$ (if any). In Theorem \ref{t:decor} we use the
following \emph{E-construction},
and we use the same notation whenever we implement it.

\noindent \textbf{The E-construction.} 
Draw a
$K$-ray $E$ landing at a periodic repelling point $x\ne f(\omega_1), x\in K^*$. 
Construct the two pullback
$K$-rays $E'$ and $E''$ of $E$ landing at distinct points $x',
x''\in K^*$ where $x'\ne x''$ by the choice of $x$.  
The set $\C\sm
(K^*\cup E'\cup E'')$ consists of components $Z_1$ and $Z_2$.
Assume that $Z_1$ contains all $K$-rays with arguments from an open arc
$I_1$ of length $1/3$ while $Z_2$ contains all $K$-rays with arguments
from the open arc $I_2$ of length $2/3$ disjoint from $I_1$. Notice
that since $x'\ne x''$ then some periodic repelling points of $K^*$
belong to $\ol{Z_1}$ and are accessible by $K$-rays from within $Z_1$
(the same claim holds for $Z_2$).

\begin{thm}\label{t:decor}
The critical decoration $D_c$ maps onto the entire $K(f)$ while any
other decoration maps onto a decoration in the one-to-one fashion. Any
decoration $D\ne D_v$ has three homeomorphic pullbacks: $D_1$
such that $\ol{D_1}\sm D_1\subset \wt K^*$ and $D_2, D_3$ that are decorations.
$D_i$ itself is a decoration for $i=2, 3$. Decoration $D_v$ has a
homeomorphic pullbacks $D'_v$ which is itself a decoration, and a
pullback $T$ that maps onto $D_v$ in the two-to-one fashion, contains
$\om_2$, is contained in $D_c$, and accumulates to both $K^*$ and $\wt
K^*$.
\end{thm}

\begin{proof}
We prove Theorem \ref{t:decor} step by step.

\noindent \textbf{Step 1.} \emph{If $D$ is a decoration of $f$, then
every pullback of $D$ is a subset of some decoration of $f$. Moreover,
if $D'$ and $D''$ are decorations and $f(D')\cap D''\ne \0$, then
$f(D')\supset D''$}.

\noindent Proof of Step 1. Let $\Gamma$ be a pullback of $D$. Clearly,
$\Gamma\subset K\sm K^*$. Since $\Gamma$ is connected, it must lie in some
decoration. Now, if $D'$ and $D''$ are decorations and $f(D')\cap
D''\ne \0$ then we can choose a pullback $D'''$ of $D''$ which is
non-disjoint from $D'$. By the above $D'''\subset D'$; by Lemma
\ref{l:pull}, \, $f(D''')=D''$. Hence $f(D')\supset D''$ as desired.

\noindent \textbf{Step 2.} \emph{If $D$ is a non-critical decoration
then $f(D)$ is a decoration.}

\noindent Proof of Step 2. The set $f(D)$ is connected and disjoint
from $K^*$ 
(by definition of a non-critical decoration). 
Hence it is
contained in one decoration. By Step 1, the set $f(D)$ coincides with this
decoration.

\noindent \textbf{Step 3.} \emph{$\om_2\in D_c$}.

\noindent Proof of Step 3. Let $\om_2\notin D_c$.  Since by Lemma
\ref{l:triv}, $A=D_c\cup K^*$ is compact, a neighborhood of $\om_2$ is
disjoint from $A$.  It follows that the set $L$ is contained in a
decoration $D\ne D_c$. Choose a neighborhood $U$ of $f(L)$ so that a
pullback $W$ of $U$ with $\om_2\in W$ is disjoint from $A$. Then $W$ is
a neighborhood of $L$ that maps two-to-one onto $U$.

Choose a $K$-ray $R$ that enters $U$ and denote the first (coming from
infinity) point of intersection of $R$ and $\bd(U)$ 
by $x$.
Denote the union of the segment of $R$ from infinity to $x$ and a curve
$I'\subset U$ from $x$ to $v_2$ by $R'$. The pullback $R''$ of $R'$
containing $\om_2$ consists of two segments of two $K$-rays each of
which maps to the segment of $R$ from infinity to $x$, and an arc $I''$
that double-covers $I'$. Clearly, $R''$ partitions the plane in two
half-planes on one of which $f$ acts in the one-to-one fashion while on
the other one it acts in the two-to-one fashion. However by the
construction $R''$ is disjoint from $A$ and points of $K^*\subset A$
have three preimages in $A$, a contradiction. Thus, $\om_2\in D_c$.

\noindent \textbf{Step 4.} \emph{We have $\om_2\in Z_2$, and, hence,
$D_c\subset Z_2$. The map $f$ is a homeomorphism on $Z_1$. If $D$ is a
decoration then $Z_1$ contains a unique homeomorphic pullback $D'$ of
$D$ which is itself a non-critical decoration. No decoration has a
unique pullback.}

\noindent Proof of Step 4. Both $Z_1$ and $Z_2$ contain points from
$K\sm K^*$; in fact, their union contains the entire $K\sm K^*$. Hence,
all the decorations are contained in $Z_1\cup Z_2$. Notice though, that
the $K$-rays approaching points of $K\sm K^*$ have arguments from the
disjoint open arcs of angles at infinity, namely, from $I_1$ and $I_2$,
respectively. Since all $K$-rays in $Z_1$ have arguments from the arc
$I_1$, then they all have distinct images. We claim that $\om_2\notin
Z_1$. Indeed, suppose that $\om_2\in Z_1$. By Step 3 then $D_c\subset
Z_1$ and so $\wt K^*\subset Z_1$. Choose a repelling periodic point
$y\in K^*$ accessible by a $K$-ray $R$ to $K^*$ from within $Z_1$. Then
choose the first preimage $y'\in \wt K^*$ of $f(y)$. Clearly, $y'$ is also
accessible by a $K$-ray $R'$ which itself is a pullback of $f(R)$. However
both rays must have arguments from $I_1$, a contradiction (recall that
$I_1$ is an arc of length $1/3$). Thus, $\om_2\in Z_2$ and, hence,
$D_c\subset Z_2$.

Let us show that $f|_{Z_1}$ is a homeomorphism. It is a homeomorphism
on 
the union of all $K$-rays with arguments
from $I_1$. Suppose that points $x, y\in Z_1$ are such that
$f(x)=f(y)=z$. We may assume that $z\notin E\cup \{v_2\}$. Since
$\wt K^*\subset D_c\subset Z_2$, then $z\notin K^*$. Hence we can construct a ray $H$
from $z$ to infinity bypassing $K^*\cup v_2\cup E$. Consider pullbacks
$H_x$ and $H_y$ of $H$ such that $x\in H_x$ and $y\in H_y$. By the
choices we made, $H_x\subset Z_1, H_y\subset Z_1,$ and $H_x\cap
H_y=\0$. It follows that there exist distinct points $x'\in H_x$ and
$y'\in H_y$ with the same image and such that both $x'$ and $y'$ belong
to $\C\sm K$. However then $x'$ and $y'$ must belong to $K$-rays with
arguments from $I_1$, a contradiction. Hence $f|_{Z_1}$ is one-to-one.
By Brouwer's Invariance of Domain Theorem, $f|_{Z_1}$ is a
homeomorphism onto $f(Z_1)$.

Let $D$ be a decoration. Choose $y\in D\cap J(f)$ 
and a sequence
$y_i\in \C\sm K$ of points that converge to $y$ and belong to $K$-rays $R_i\ne
E$. Choose $K$-rays $R'_i\in Z_1$ with $f(R'_i)=R_i$ and then points
$y'_i\in R'_i$ such that $f(y_i')=y_i$. Then $y'_i\to y'$ where
$f(y')=y$, and hence $y'\in D'$ where $D'\subset Z_1$ is a decoration.
The rest of the claim follows.

\noindent \textbf{Step 5.}  \emph{If $D\ne D_v$ is a decoration  then
it has three pullbacks each of which maps onto $D$ homeomorphically.
Two of the pullbacks accumulate into $K^*$; one pullback accumulates
into $\wt K^*$.}

\noindent Proof of Step 5. By Lemma \ref{l:triv} the set $K^*\cup D$ is
a full continuum; clearly, $v_2\notin K^*\cup D$. Apply the E-construction
to $D$. Then draw a ray $R$ from $v_2$ to infinity so that $R\cap
(K^*\cup D\cup E)=\0$. Construct the pullback $C$ of $R$ passing
through $\om_2$; the set $C$ cuts the plane in two half-planes, $X$ and $Y$,
such that $X$ maps onto $\C\sm R$ in the one-to-one fashion while $Y$
maps onto $\C\sm R$ in the two-to-one fashion. The cut $C$ is disjoint
from $K^*\cup \wt K^*$ as well as from the pullbacks of $D$. Hence $C$
separates $K^*$ and $\wt K^*$; it is easy to see that $\wt K^*\subset
X$ and $K^*\subset Y$. It follows that there exists a unique pullback
$D_1$ of $D$ contained in $X$. Since it has to accumulate to points
mapped to $K^*$, the set $D_1$ accumulates into $\wt K^*$. Hence
$D_1\subset D_c$. Since $D_c\subset Z_2$, then $D_1\subset Z_2$.

In addition to $D_1$, by Lemma \ref{l:pull} there may be either one
pullback of $D$ mapped onto $D$ in the two-to-one fashion, or two
pullbacks of $D$ mapped onto $D$ in the one-to-one fashion. By Step 4,
$D$ has a homeomorphic pullback $D_2\subset Z_1$ which is a non-critical
decoration. Clearly, $D_2\ne D_1$. Hence $D$ has three homeomorphic
pullbacks of which $D_1$ accumulates to $\wt K^*$ and $D_2$ accumulates
to $K^*$. Let $D_3$ be the remaining pullback of $D$. If it accumulated
to $\wt K^*$ it would follow that $f|_{\wt K^*}$ is not one-to-one, a
contradiction. Hence $S$ accumulates to $K^*$.

\noindent \textbf{Step 6.} \emph{The decoration $D_v$ has two
pullbacks. One of them, say, $T$, maps onto $D_v$ in the two-to-one
fashion, contains $\om_2$, is contained in $D_c$, and accumulates to
both $K^*$ and $\wt K^*$; the other one is the homeomorphic pullback
$D'_v$ defined in Step 4.}

\noindent Proof of Step 6. Clearly, $D_v$ cannot have three pullbacks
as otherwise the point $v_2$ will have four preimages (counted with
multiplicity), a contradiction. By Step 4, 
$D_v$ has a
homeomorphic pullback $D'_v\subset Z_1$. Thus, we only need to study
the remaining two-to-one pullback $T$ of $D_v$. Clearly, $\om_2\in T$
which  implies that $T\subset D_c$. To prove that $T$ accumulates in
both $K^*$ and $\wt K^*$, notice that there are points of $T$ close to
$\wt K^*$. Indeed, a neighborhood of $\wt K^*$ maps homeomorphically
onto a neighborhood of $K^*$ and therefore contains points of the
preimage of $D_v$. These points cannot belong to $D'_v$ because
$\ol{D'_v}\sm D'_v\subset K^*$ by Lemma \ref{l:triv}. Hence they belong to
$T$ as claimed. Thus, $T$ accumulates into $\wt K^*$. To show that $T$
accumulates into $K^*$ too, choose a sequence of points $y_i\in D_v$
such that $y_i\to y\in K^*$. For each $i$, let $y'_i, y''_i\in T$ be
two distinct preimages of $y_i$ in $T$. We may assume that they
converge to $y', y''$ respectively. If $y', y''\notin K^*$ then $y',
y''\in \wt K^*$ which implies that in any neighborhood of $\wt K^*$
there are pairs of points with the same image, a contradiction. Hence
$T$ accumulates into both $K^*$ and  $\wt K^*$.
\end{proof}

\section{Quadratic arguments}
\label{s:quadarg}

Consider $K^*$-rays $R^*(\al)$. Clearly, $f(R^*(\al))\supset R^*(2\al)$
(the curve $f(R^*(\al))$ extends the ray $R^*(2\al)$ into the annulus
between the basic neighborhood of $K^*$ and its image). A \emph{crosscut}
(of $K^*$) is a closed arc $I$ 
with endpoints $x, y\in K^*$ such that
$[I\sm \{a, b\}]\subset \C\sm K^*$. If $a_n$ is a crosscut then
the \emph{shadow} $\cont(a_n)$ of a
crosscut $a_n$ is the bounded complementary component of $a_n\cup K^*$.
A sequence of crosscuts $a_n, n=1, 2, \dots$ is \emph{fundamental} if
$a_{n+1}\subset \cont(a_n)$ for every $n$ and the 
diameter of $a_n$
converges to $0$ as $n\to\infty$. Two fundamental sequences of crosscuts
are \emph{equivalent} if crosscuts of one sequence are eventually
contained in the shadows of crosscuts of the other one, and vice versa.
This is an equivalence relation whose classes are called \emph{prime ends}.
In what follows the set of endpoints of a closed arc $I$ is denoted by $\en(I)$.

By the Carath\'eodory theory, every quadratic-like ray $R^*(\al)$ to $K^*$
corresponds to a certain prime end $E^*(\al)$ represented by a fundamental sequence
of crosscuts $\{a_n\}$. Moreover,
for every $a_n$ a \emph{tail} of $R^*(\al)$ is contained in $\cont(a_n)$ (a
\emph{tail} of $R^*(\al)$ is defined by a point $x\in R^*(\al)$ and
coincides with the component of $R^*(\al)\sm \{x\}$ that accumulates
into $K^*$). It is convenient to consider also the associated picture in
$\C\sm \disk$; the picture on the $K^*$-plane is transferred to that in
$\C\sm \disk$ by means of the map $\psi^*$ introduced earlier.

Namely, for every $n$, the set $\psi^*(a_n\sm \en(a_n))$ is an arc
$I_n\subset \C\sm \cdisk$ without endpoints such that $\ol{I}$ is a closed ark
with endpoints $x_n, y_n\in \uc$. One can choose
the circle arc $I'_n$ positively oriented from $x'_n$ to $y'_n$ such that
$\al\in I'_n$ and consider the Jordan curve $Q_n=I_n\cup
I'_n$; then the radial ray $R_\al$ with the initial point
at $z_\al\in \uc$ where $z_\al\in \uc$ is the point of the circle with argument
$\al$, intersected with the simply connected domain $U_{a_n}$ with boundary
$Q_n$, contains a small subsegment of $R_\al$ with an endpoint $z_\al$.
Observe that $(\psi^*)^{-1}(U_{a_n})$ is the shadow of the crosscut $a_n$.
The \emph{impression} of $E^*(\al)$ is the intersection of the closures of $\cont(a_n)$. We
say that a prime end $E^*(\al)$ is \emph{disjoint} from a set $S\subset\C\sm
K^*$ if $\cont(a_n)\cap S=\0$ for all sufficiently large $n$.
Riemann map of $\C\sm K^*$. In what follows, when talking about crosscuts, we will use this notation.

\begin{lem}\label{l:cross-primends}
Suppose that $X\subset \C\sm K^*$ is a connected set, and
$\psi^*(X)$ accumulates on exactly one point $z_\al\in \uc$ with
argument $\al$. Then $X$ is non-disjoint from $E^*(\al)$ and
disjoint from any other prime end.
\end{lem}

\begin{proof}
Let $a$ be a crosscut associated with $E^*(\al)$. Consider the set $U_a$. Since
$\psi^*(X)$ accumulates on $z_\al$, then $\psi^*(X)$ is non-disjoint
from $U$, and hence $X$ is non-disjoint from $\cont(a)$. By definition, $X$ is
non-disjoint from  $E^*(\al)$.
Also, for any point $t=e^{2\pi \beta i}\ne z_\al$ we can find a crosscut
$b$ associated with $\beta$ and so small that $\ol{U_b}$ is
disjoint from $\psi^*(X)$. Then $X$ is disjoint from
$\cont(b)$ and hence $X$ is disjoint from $E^*(\beta)$.
\end{proof}

Recall that a set $A$ \emph{accumulates in} $B$ if $\ol{A}\sm A\subset B$.

\begin{lem}
  \label{l:accum}
Suppose that a connected subset $\wh D\subset\C\sm\ol\disk$ accumulates on
a non-degenerate arc $A\subsetneqq \uc$. Let $\wh R$ be an arc in
$\C\sm\ol\disk$ landing at an interior point $y$ of $A$. Then $\wh R$
crosses $\wh D$.
\end{lem}

\begin{proof}
By way of contradiction assume that $\wh R_1\cap \wh D=\0$. Extend $\wh R_1$ to
infinity still avoiding $\wh D$. Since $A\ne \uc$, then there exists a ray
$\wh R_2$ from a point $x\in \uc\sm A$ to infinity that is disjoint from
$\wh D$. It follows that $\wh D$ is contained in a component of $\C\sm
(\wh R_1\cup \wh R_2\cup I)$ where $I$ is one of the two circle arcs with
endpoints $x$ and $y$. However then $\wh D$ can only accumulate on the
part of $A$ that is contained in $I$, a
contradiction.
\end{proof}

Proposition \ref{p:pends} uses Lemma \ref{l:accum}.

\begin{prop} \label{p:pends}
Every decoration $D$ is disjoint from all prime ends of $K^*$ except
exactly one.
\end{prop}

\begin{proof}
The (connected) set $D'=\psi^*(D)$ accumulates to a closed arc
$A\subset \uc$. Suppose that $A$ is non-degenerate, and bring this to a
contradiction. First, for every repelling periodic point $x\in K$ there
exists an $K$-ray landing at $x$. Also, $K$-rays are disjoint from $D$.
Choose two distinct repelling periodic points in $K^*$, draw $K$-rays
landing at them, and then map all this by $\psi^*$ to
$\ol\C\sm\cdisk$. This will result in two curves landing at two
distinct points of $\uc$ and disjoint from $D'$. Thus, $D'$ does not
accumulate onto the entire $\uc$.

Now, choose a periodic $K$-ray $R$ that lands at a repelling periodic
point $w\in K^*$. Under the map $\psi^*$ it is associated to a curve
$\psi^*(R)$ that lands at a certain point $x\in \uc$ and is
otherwise disjoint from $\cdisk$. Pulling back $R$ under $f$ following
backward orbit of $w$ in $K^*$ corresponds to pulling back
$\psi^*(R)$ under $z^2$. Since $A$ is non-degenerate, there exists a
number $N$ such that some $N$-th pullback of $R$ lands at a point
$w'\in K^*$ while the corresponding $N$-th pullback of $\psi^*(R)$
is a curve $Q$ in $\C\sm \cdisk$ that lands at an interior point of
$A$. Since $Q$ is disjoint from $D'$, we see by Lemma \ref{l:accum}
that $A=\{\al\}$ is degenerate. By Lemma \ref{l:cross-primends}, $D$ is
disjoint from all prime ends of $K^*$ except for $E^*(\al)$.
\end{proof}

Suppose that $E^*(\al)$ is the only prime end non-disjoint from $D$.
Then $\al$ is called the \emph{quadratic argument} of $D$ and is
denoted by $\al(D)$. By Proposition \ref{p:pends}  each decoration has
only one quadratic argument (and so quadratic arguments are well
defined) while different decorations may a priori have the same
quadratic argument. A useful interpretation of these concepts is as
follows. Using the map $\psi^*$ for $\C\sm K^*$ we can transfer
all decorations to the set $\C\sm \cdisk$; then for any decoration $D$
the set $\psi^*(D)$ accumulates to the point $e^{2\pi i \al(D)}$ of the
unit circle with argument $\al(D)$ (i.e., $\ol{\psi^*(D)}\sm \psi^*(D)$
consists of one point from the unit circle with argument $\al(D)$).
Moreover, if $U^*$ is a basic neighborhood of $K^*$ then $\psi^*$
conjugates $z^2$ restricted onto $\psi^*(U^*\sm K^*)$ and $f$
restricted onto $U^*\sm K^*$.

The dynamics on the uniformization plane immediately implies the next lemma stated here
without proof.

\begin{lem}\label{l:z2}
If $D\ne D_c$ is a decoration then $\al(f(D))=\si_2(\al(D))$. On the other hand,
in the notation of Theorem $\ref{t:decor}$, we have that $\al(D_v)=\si_2(\al(D_c))$.
\end{lem}

In what follows we use Riemann maps and their inverses for either the entire filled Julia
set $K$ of $f$, or a quadratic-like Julia set $K^*$. In the former case
we talk about $K$-plane and $z^3$-plane (which are associated to each
other under the appropriate Riemann map), in the latter case,
similarly, we talk about $K^*$-plane and $z^2$-plane.

Corollary \ref{c:2pbdec} follows from Proposition \ref{p:pends}.

\begin{cor} \label{c:2pbdec}
Consider a decoration $D$ with quadratic argument $\al$. Then both
$\al/2$ and $(\al+1)/2$ are quadratic arguments of decorations
containing pullbacks of $D$. These two decorations are different.
\end{cor}

\section{Cubic arguments}\label{s:cubarg}

So far we have been working on establishing general facts concerning
the situation when a cubic polynomial $f$ has a connected
quadratic-like filled Julia set $K^*$ (clearly, $K^*\subset K(f)$). In this section we begin looking into more specific cases;
as the first step we describe some results and concepts, mostly taken from
\cite{bot16}.

Consider an immediately renormalizable polynomial $f\in \Fc_\la$ with
$|\la|\le 1$, and define an invariant quadratic gap $\Uf(f)$ associated
with $0$; as before, by $f^*:U^*\to V^*$, we denote the corresponding
quadratic-like map. When $J(f)$ is disconnected, gaps analogous to
$\Uf(f)$ were studied in \cite{bclos} where tools developed in
\cite{lp96} were used; however this approach is based upon the fact
that $J(f)$ is disconnected in an essential way and, hence, is very
different from that used in \cite{bot16} and here. Once we introduce
$\Uf(f)$, we shall see that it coincides with the similar gap in the
disconnected case. Recall that by Lemma \ref{l:7.2} if $f$ is a cubic
polynomial with a non-repelling fixed point $a$, then there exists at
most one quadratic-like filled invariant Julia set $K^*$ containing
$a$; by Corollary \ref{c:7.2} if $f$ has empty rational lamination then
it has a unique non-repelling fixed point $a$ and at most one
quadratic-like filled invariant Julia set that, if it exists, must
contain $a$.

To define the gap $\Uf(f)$ associated with $K^*$, we use (pre)periodic
points of $f$. Since $K^*$ is a quadratic-like filled Julia set, then $K^*$ is
a component of $f^{-1}(K^*)$.

\begin{dfn}\label{d:uf}
Let $\widehat X(f)=\widehat X$ be the set of all $\si_3$-(pre)periodic
points $\ol\al\in\uc$ such that $R_f(\al)$ lands in $K^*$. Let $X(f)=X$
be the closure of $\widehat X$. Let $\Uf(f)$ be the convex hull of $X$.
Let $\widetilde K^*$ be the component of $f^{-1}(K^*)$ different from
$K^*$ (such a component of $f^{-1}(K^*)$ exists because $f|_{K^*}$ is
two-to-one). Let $Y(f)=Y$ be the closure of the set of all preperiodic
points $\ol\al\in\uc$ with $R_f(\al)$ landing in $\widetilde K^*$.
Observe, that $\widetilde K^*$ is disjoint from $U^*$ (otherwise points
of $f(\widetilde K^*\cap U^*)$ must belong to $K^*$, a contradiction with dynamics
of points of $\widetilde K^*\cap U^*$).
\end{dfn}

From now on we fix an immediately renormalizable
polynomial $f\in \Fc_{nr}$ and do not refer to $f$ in our notation
(we write $\Uf$ instead of $\Uf(f)$ etc). Lemma \ref{l:inv-quad} summarizes
some results of Section 7 of \cite{bot16}.

\begin{lem}\label{l:inv-quad}
The set $\widetilde K^*$ is disjoint from $U^*$.
The gap $\Uf$ is an invariant quadratic gap of regular critical or periodic type.
The map $\si_3|_{\widehat X}$ is two-to-one, and
$Y$ lies in the closure of the major hole of $X$.
\end{lem}

This shows that the results of \cite{BOPT} and \cite{bot16} apply to $\Uf$
(recall that these results are described in Subsection \ref{ss:stru-fla}).
E.g., consider the map $\tau$ defined there for any quadratic invariant gap of $\si_3$
($\tau$ collapses edges of $\Uf$ and semiconjugates
$\si_3|_{\bd(\Uf)}$ and $\si_2$). The map $\psi^*$ maps $K$-rays to their counterparts on the
$z^2$-plane. On the other hand, the Riemann map defined by $K$ sends
the radial rays with rational arguments in $\C\sm \cdisk$, to $K$-rays,
including $K$-rays landing in $K^*$. Composing
these two maps we obtain a map $\eta$ that associates
radial rays from the $z^3$-plane with arguments from $\wh X$ to
$\psi^*$-images of $K$-rays on the  $z^2$-plane landing in $\uc$. Thus, if $R(\be)$ is a radial ray on
$z^3$-plane and $\be\in \widehat X$, then the ray $\eta(R(\be))$ is
contained in $z^2$-plane and lands at a point of $\uc$; this defines a
map $\tau':\widehat X\to \uc$. By construction we see that $\tau'$
semiconjugates $\si_3|_{\widehat X}$ and $\si_2$; the uniqueness of the
map $\tau$ shows that $\tau'$ is the restriction of $\tau$ onto $\widehat X$.

Observe that $K$-rays with arguments from $\bd(\Uf)$ do not necessarily
have principal sets contained in $K^*$. Nevertheless the map $\tau$
allows us to relate decorations of $K^*$ and their quadratic arguments
with edges and vertices of the gap $\Uf$.

\begin{lem}\label{l:cridi3}
The quadratic argument of $D_c$ is $\tau(M_{\Uf})$.
\end{lem}

\begin{proof}
If the quadratic argument of $D_c$ is not
$\tau(M_{\Uf})$, then there is an edge/vertex $v$ of $\Uf$
such that $\tau(v)$ is the quadratic argument of $D_c$ and
we can find angles $\al, \be\in \wh X$ such that the arc
$I=(\al, \be)$ contains $v$ but does not contain the endpoints of
$M_{\Uf}$. For $K$-rays $R(\al), R(\be)$ with arguments $\al, \be$, consider the
component $W$ of $\C\sm [R(\al)\cup R(\be)\cup K^*]$ containing
$K$-rays with arguments from $I$. By Lemma \ref{l:inv-quad} $Y$
lies in the closure of the major hole of $X$. Hence $\widetilde K(f^*)$
is disjoint from $W$ despite the fact that $\wt K(f^*)\subset D_c\subset W$.
\end{proof}


In the end of this section we study the issue of landing at points of $K^*$ for
periodic and preperiodic angles from $\Uf$.


\begin{lem}\label{l:who-land}
Let $\al\in \Uf$ be a (pre)periodic angle that never maps to an
endpoint of the major $M_{\Uf}$ of $\Uf$. Suppose that the $K$-ray $R(\al)$ with
argument $\al$ lands at a point $x$. Then $x\in K^*$.
\end{lem}

In our notation the claim of the lemma simply means that $\al\in \widehat X$.

\begin{proof}
Let $\al=0\in \Uf$, yet $x\notin K^*$.
Let $y\in K^*$ be the fixed point at which $R^*(0)$ lands. Then the only $K$-ray that can land
at $y$ is $R(\frac12)$ which implies that $M_{\Uf}=\ol{0 \frac12}$, a contradiction with the
assumptions of the lemma. Similarly, if $\al=\frac12$, then $x\in K^*$.

We claim that the lemma holds for a (pre)periodic angle $\al\in \Uf$ if it holds for $\be=\si_3(\al)$.
By way of contradiction assume that $x\notin K^*$; then $x\in \wt K^*$, $\al\in Y$ and, by Lemma \ref{l:inv-quad},
$Y$ is contained in the closure of the major hole of $\Uf$. It follows that $\al$ is an endpoint of $M(\Uf)$,
a contradiction with the assumptions of the lemma. By the first paragraph we conclude that the lemma holds
if $\al$ eventually maps to a $\si_3$-fixed angle.

We claim that if $\be\in \Uf\cap \widehat X$ and $\si_3(\al)=\be$ then $\al\in \widehat X$. Indeed,
by Lemma \ref{l:inv-quad} $\si_3|_{\widehat X}$ is two-to-one. If $\al\notin \widehat X$ then all three
$\si_3$-preimages of $\be$ belong to $\Uf$. However this is only possible if two of them are endpoints
of $M_{\Uf}$. It follows that $\al$ is an endpoint of $M_{\Uf}$, again a contradiction.
This and the first paragraph of the proof imply the claim of the lemma if $\al$ eventually maps to a $\si_3$-fixed
angle.

Assume now that $\al$ never maps to a $\si_3$-fixed angle. The angle $\tau(\al)$ is (pre)periodic under $\si_2$
and never maps to $0$ under iterations of $\si_2$.
The $K^*$-ray $R^*(\tau(\al))$ with quadratic argument $\tau(\al)$ lands at a point $x'\in K^*$. Let a $K$-ray $R(\gamma)$
land at $x'$. We claim that $\al=\ga$. Indeed, $\al\in \Uf$ by the assumptions, and $\ga\in \Uf$
since $R(\gamma)$ lands at $x'\in K^*$. Moreover, neither $\al$ nor $\ga$ ever map to $0$ or $\frac12$. It now follows
from the construction that both angles (recall that $\al$ and $\ga$ belong to $\Uf$) behave in the same fashion with respect
to the partition of $\Uf$ in two arcs by the $\si_3$ point that belongs to $\Uf$, and its preimage in $\Uf$ (or, in
the appropriate exceptional cases, by the major $M(\Uf)=\ol{0 \frac12}$ and its preimage-edge in $\Uf$).
Together with the assumptions of the lemma that $\al$ never maps to an
endpoint of the major $M_{\Uf}$ this implies that $\al=\ga$
and hence the landing point $x$ of $R(\al)$ belongs to $K^*$ as desired.
\end{proof}

\section{Sectors}

Suppose that $f$ is an immediately renormalizable cubic polynomial.
Define a few objects depending on $f$ and denote them with
$f$ as the subscript; yet, if $f$ is fixed, we may omit it from notation.
Consider a pair of external rays $R(\al)$, $R(\be)$ landing in $K^*$.
The set $\Sigma_f=K^*\cup R(\al)\cup R(\be)$ divides the plane into two components, one
of which contains all external rays with arguments in $(\al,\be)$ and the other contains all external rays with arguments in $(\be,\al)$.
To formally justify this claim, collapse $K^*$ to a point (i.e., consider the equivalence relation $\sim$
on $\ol\C$, whose classes are $K^*$ and single points in $\ol\C\sm K^*$).
By Moore's theorem, the quotient space $\ol\C/\sim$ is homeomorphic to the sphere.
The image of $\Sigma_f(\al, \be)$ under the quotient map, together with the image of the point at infinity, form a Jordan curve.
The statement now follows from the Jordan curve theorem.
Let $S^\circ(\al,\be)_f$ be the component of $\C\sm\Sigma_f(\al,\be)$ containing all external rays with arguments in $(\al,\be)$.
Observe that $S^\circ(\al,\be)$ is defined only if the rays $R(\al)$, $R(\be)$ both land in $K^*$.
The sets $S^\circ(\al,\be)$ will be called \emph{open sectors}, and the sets $\Sigma(\al,\be)$ will be called \emph{cuts}.
Images of sectors contain $K^*$ iff sectors contain $\wt K^*$.

\def\arg{\mathrm{arg}}

An open sector $S^\circ(\al,\be)$ is associated with its \emph{argument arc} $(\al,\be)\subset\R/\Z$
that consists of arguments of all rays included in $S^\circ(\al,\be)$.
Note, 
that 
this sector does not have to coincide with the union of those rays as
open sectors may contain 
decorations. More generally, consider a subset $T\subset\C$.
We call the set $T_f$ \emph{$(f)$-radial} 
if any ray intersecting $T$ lies in $T$.
For a radial set $T$ we can define the \emph{argument set} $\arg(T)$ of $T$ as the set of all $\ga\in\R/\Z$ with $R(\ga)\subset T$.
Every open sector is a radial set, whose argument set is an open arc. 
It is clear that, for any radial set $T$, we have
$$
\arg(f(T))=\si_3(\arg(T)),\quad \arg(f^{-1}(T))=\si_3^{-1}(\arg(T)).
$$

The following properties of open sectors are almost immediate.

\begin{lem}
  \label{l:pb-opsec}
 Let $S^\circ$ be an open sector and $T^\circ$ an $f$-pullback of $S^\circ$.
Then $\arg(T^\circ)$ is the union of one or several components of $\si_3^{-1}(\arg(S^\circ))$.
The number of critical points in $T^\circ$ equals the number of components minus one.
If $\om_2\notin T^\circ$ and the closure of $T^\circ$ intersects $K^*$, then $T^\circ$ is an open sector mapping 1-1 onto $S^\circ$.
Any pullback of $S^\circ$ is disjoint from $K^*$.
\end{lem}

\begin{proof}
The first claim ($\arg(T^\circ)$ is a union components of $\si_3^{-1}(\arg(S^\circ))$) is immediate.
Note that $f:T^\circ\to S^\circ$ is proper.
Therefore, this map has a well-defined degree.
The degree is clearly equal to the number of components in $\arg(T^\circ)$.
On the other hand, by the Riemann--Hurwitz formula, the degree is equal to the number of critical points in $T^\circ$ plus one.
Thus the second claim follows.

Let us prove the third claim of the lemma.
The only critical point that can lie in $T^\circ$ is $\om_2$.
Since we assume that $\om_2\notin T^\circ$, then $\arg(T^\circ)$ has only one component.
Let $\arg(T^\circ)=(\al,\be)$.
Then $T^\circ$ is bounded by $R(\al)\cup R(\be)$ and 
a part of $K^*\cup \wt K^*$.
If both $R(\al)$ and $R(\be)$ land in $K^*$, then, by definition, $T^\circ$ coincides with the open sector $S^\circ(\al,\be)$.
If both $R(\al)$ and $R(\be)$ land in $\wt K^*$, then $T^\circ$ is disjoint from $K^*$
as its image $S^\circ$ does not contain $K^*$; this implies that $\ol{T^\circ}$ is disjoint with $K^*$, a contradiction
with our assumptions.
Finally, if one of the two rays $R(\al)$, $R(\be)$ lands in $K^*$ and the other lands in $\wt K^*$, then we get a contradiction too.
Suppose, say, that $R(\al)$ lands in $K^*$ and $R(\be)$ lands in $\wt K^*$.
The sets $K^*\cup R(\al)$ and $\wt K^*\cup R(\be)$ are closed disjoint non-separating sets, whose union cannot separate the plane.
Thus the only possibility is that $T^\circ=S^\circ(\al,\be)$ and $f:T^\circ\to S^\circ$ is a homeomorphism.
The last claim of the lemma now easily follows. 
\end{proof}

\begin{lem}
  \label{l:img-sec}
  Consider an open sector $S^\circ(\al,\be)$, whose argument arc is mapped one-to-one under $\si_3$.
Then $f(S^\circ(\al,\be))=S^\circ(3\al,3\be)$.
Moreover, $S^\circ(\al,\be)$ maps one-to-one onto $S^\circ(3\al,3\be)$.
\end{lem}

\begin{proof}
  Let $T^\circ$ be the $f$-pullback of $S^\circ(3\al,3\be)$ that includes rays with arguments in $(\al,\be)$.
Clearly, the rays $R(\al)$, $R(\be)$ are on the boundary of $T^\circ$.
Since these rays land in $K^*$ and $T^\circ\cap K^*=\0$, we must have $T^\circ\subset S^\circ(\al,\be)$.
On the other hand, if $T^\circ\ne S^\circ(\al,\be)$, then, by Lemma \ref{l:pb-opsec},
 the arguments of rays in $T^\circ$ form two intervals of $\R/\Z$ rather than one.
A contradiction.
Therefore, $T^\circ=S^\circ(\al,\be)$ as desired.
\end{proof}

\begin{dfn}[minimal sectors]
  Let $x$ be a point outside of $K^*$.
The \emph{minimal sector} $S(x)$ of $x$ is defined as the intersection of all $\ol {S^\circ}(\al,\be)$ such that $x\in S^\circ(\al,\be)$.
\end{dfn}

Note that, by definition, a minimal sector is always a closed set. It is clear
from the definition that a minimal sector is bounded by at most two external rays and a piece of $K^*$.
It may coincide with the union of a single ray and its impression. The next lemma shows that
minimal sectors are related to decorations and immediately follows from the definitions. Recall that if $x\notin K^*$ then $D(x)$ is the decoration
containing $x$. Recall also that the map $\tau:\Uf\to \uc$ collapses to points all edges
of the gap $\Uf$ and semiconjugates $\si_3|_{\bd(\Uf)}$ and $\si_2$.

\begin{lem}\label{l:dec-sec}
Let $x\notin K^*$. Consider the edge (possibly degenerate) $\ol{ab}=\tau^{-1}(\al(D(x)))$ of $\Uf$. Then
$S(x)\sm K^*$ is the union of all decorations with quadratic argument $\al(D(x))$ and all $K$-rays with arguments
from the complementary arc of $\bd(\Uf)$ in $\uc$ with endpoints $a$ and $b$.
\end{lem}

Define the \emph{critical sector} as the minimal sector $S(\om_2)$.

\begin{lem}
 \label{l:noncr-sec}
 Suppose that $\om_2\notin S(x)$.
Then $f(S(x))=S(f(x))$.
\end{lem}

\begin{proof}
We first prove that $f(S(x))\subset S(f(x))$.
Indeed, $S(x)$ is the intersection of all $\ol{S^\circ}(\al,\be)$ with $x\in S^\circ(\al,\be)$.
The $f$-image of the intersection lies in the intersection of images.
Taking only those $S^\circ(\al,\be)$, for which $|\be-\al|<1/3$, we see by Lemma \ref{l:img-sec} that
 $f(S(x))$ lies in the intersection of $\ol{S^\circ}(3\al,3\be)\ni f(x)$.
The latter set obviously coincides with $S(f(x))$.

We now prove that $S(f(x))\subset f(S(x))$, i.e., every point $z$ in $S(f(x))$ has the form $f(y)$ for some $y\in S(x)$.
Take an open sector $S^\circ(\al,\be)\ni x$ that contains only one $f$-preimage of $z$; call this preimage $y$.
We may also assume that $\om_2\notin S^\circ(\al,\be)$ and that $(\al,\be)$ maps one-to-one under $\si_3$.
Then, by Lemma \ref{l:img-sec}, we have $f(S^\circ(\al,\be))=S^\circ(3\al,3\be)$.
Let $S^\circ$ be any open sector in $S^\circ(3\al,3\be)$ containing $z$.
Then the pullback $T^\circ$ of $S^\circ$ in $S^\circ(\al,\be)$ must be an open sector, and it must contain a preimage of $z$.
The only option is that $y\in T^\circ$.
Since $y$ is contained in all such $T^\circ$, we have $y\in S(x)$.
\end{proof}

\begin{lem}
  \label{l:bdminsec}
  For any $x\in\C\sm K^*$, the rays on the boundary of $S(x)$ map onto the rays on the boundary of $S(f(x))$ under $f$.
\end{lem}

\begin{proof}
  Consider an open sector $S^\circ$ around $x$.
Its argument arc covers one or several components of $\si_3^{-1}(\arg(S(f(x))))$.
Moreover, $S^\circ$ can be chosen so that the endpoints of $\arg(S^\circ)$ are arbitrarily close to  $\si_3^{-1}(\arg(S(f(x))))$.
The lemma follows.
\end{proof}

\section{Backward stability}

In this section we study backward stability of decorations. The aim, to
begin with, is to show that under certain circumstances decorations
shrink as we pull them back.
Our arguments are based upon the following theorem of Ma\~n\'e.

\begin{thm}[\cite{man93}]
\label{t:mane} If $f:\ol\C\to\ol\C$ is a rational map and $z\in\ol\C$ a
point that does not belong to the limit set of any recurrent critical
point, then for some Jordan disk $W$ around $z$, some $C>0$ and some
$0<q<1$ the spherical diameter of any component of $f^{-n}(W)$ is less
than $Cq^n$.
\end{thm}

Neighborhoods satisfying Theorem \ref{t:mane} are called \emph{Ma\~ne neighborhoods}.

\begin{lem}
  \label{l:sn}
  Fix $q\in (0,1)$ and $b>0$.
  Consider a sequence of positive numbers $s_n$ such that either $s_{n+1}=qs_n$ or $s_{n+1}\le 2qs_n+b$.
In the former case call $n$ the \emph{good} index, and in the latter case call $n$ the \emph{bad} index.
Suppose that the distance between adjacent bad indices tends to infinity.
Then $s_n\to 0$. 
\end{lem}

\begin{proof}
It suffices to show that $s_n\to 0$ as $n$ runs through all bad indices $n_1<n_2<\dots$;
fix $\eps>0$ and $N$ such that $q^N<1/8$ and $q^{N-1}b<\eps$.
Then, for $i$ large, we have $n_{i+1}-n_i\ge N$, and
$$
s_{n_{i+1}}=q^{n_{i+1}-n_i-1}(2qs_{n_i}+b)=q^{n_{i+1}-n_i}(2 s_{n_i}+q^{-1}b)\le \frac{s_{n_i}}{4}+\eps.
$$
Since the map $h(x)=x/4+\eps$ has a unique attracting point $4\eps/3$ which attracts all points of $\R$, then
$s_{n_i}$ becomes eventually less than $4\eps$.
Since $\eps>0$ is arbitrary, it follows that $s_n\to 0$ as $i\to\infty$, as desired.
\end{proof}

Recall that we consider immediately renormalizable polynomials $f$ and use for them
the notation introduced above.
Consider two rays $R=R(\al)$ and $L=R(\be)$ landing in $K^*$.
Also, take any equipotential $\wh E$ of $P$.
Let $\Delta=\Delta(R,L,\wh E)$ be the bounded complementary component of $K^*\cup R\cup L\cup \wh E$
such that the external rays that penetrate into $\Delta$ have arguments that belong to the positively oriented arc from $\al$
to $\be$.
Evidently, $\Delta$ is the intersection of $S^\circ(\al,\be)$ with the Jordan disk enclosed by $\wh E$.
Hence results of the previous section dealing with sectors apply to $\Delta$ and similar sets.
Let $\Delta'$ be a pullback of $\Delta$ such that $\ol{\Delta'}\cap K^*\ne\0$; then say that $\Delta'$ is a pullback
of $\Delta$ \emph{adjacent} to $K^*$. If $\Delta'$ is a $P^n$-pullback of $\Delta$ such that $\ol{\Delta'}\cap K^*\ne\0$,
we call $\Delta'$ an \emph{iterated} pullback of $\Delta$ \emph{adjacent} to $K^*$.
Let $\Delta=\Delta_0$ and for every $n$ the set $\Delta_n$ be a pullback of $\Delta_{n-1}$ adjacent to $K^*$.
Then the sequence of sets $\Delta_n, n=0, 1, \dots$ is called a \emph{backward pullback orbit of $\Delta$ adjacent to $K^*$}.
Set $\wt U^*$ to be the pullback of $U^*$ containing $\wt K^*$.

\begin{lem}
\label{l:Del}
Suppose that $\{\Delta_n\}$ is a backward pullback orbit of $\Delta$ adjacent to $K^*$ and
$n_1<n_2<\dots$ are \emph{all} positive integers $n$ such that $\om_2\in\Delta_n$
and $\om_2$ is non-recurrent.
If $n_{i+1}-n_i\to\infty$, then $\Delta_{n}\subset U^*$ for large $n$ so that the distance between $\Delta_n$ and $K^*$ tends to $0$.
\end{lem}

\begin{proof}
  Consider a finite covering $\Uc$ of $\Delta\sm U^*$ by Ma\~ne neighborhoods.
After adjustments we may assume that elements of $\Uc$ are subsets of $\Delta$. 
Similarly, fix a finite covering $\Vc$ of $\wt U^*$
by Ma\~ne neighborhoods such that $\bigcup\Vc=\wt U^*$.
Set $\Uc_1=\Uc$ and define $\Uc_n$ inductively as follows.
Assuming by induction that $\Delta_n\sm U^*\subset\bigcup\Uc_n\subset \Delta_n\cup U^*$,
 define $\Uc_{n+1}$ as the set of all open sets $U$ satisfying one of the following:
\begin{enumerate}
  \item there is $U'\in\Uc_n$ such that $U\subset\Delta_{n+1}$ is a pullback of $U'$;
  \item the point $\om_2$ is in $\Delta_{n+1}$, and $U\in\Vc$.
\end{enumerate}
Neighborhoods in $\Uc_{n+1}$ as in item $(1)$ (resp., $(2)$) are called \emph{type $(1)$} (resp., \emph{type $(2)$}) neighborhoods.

Any $U\in\Uc_n$ is either obtained as a $P^k$-pullback of some type $(1)$ neighborhood in $\Uc_{n-k}$
with $k$ being maximal with this property, or comes from $\Vc$ but only at the moments when $\om_2\in \Delta_n$.
In the former case set $s(U)=Cq^k$; by the Ma\~ne theorem, $\mathrm{diam}(U)\le s(U)$. In the latter case set $s(U)=\mathrm{diam}(U)$.
Define $s_n=s_n(\Delta)$ as the sum of $s(U)$ over all $U\in\Uc_n$.
By the triangle inequality $\mathrm{diam}(\Delta_n\sm U^*)$ is bounded above by $s_n$.
Thus it suffices to show that $s_n\to 0$ as $n\to\infty$. Now consider two cases of transition from $n$ to $n+1$.

(1) Assume that $\om_2\notin \Delta_{n+1}$.
Then $\Delta_{n+1}$ maps one-to-one to $\Delta_n$; moreover, no point of $\Delta_{n+1}\sm U^*$ maps into $U^*$.
It follows that all neighborhoods in $\Uc_{n+1}$ are type $(1)$.
In this case we have $s_{n+1}=q s_n$.

(2) Assume that $\om_2\in\Delta_{n+1}$.
Then there are at most twice as many type $(1)$ neighborhoods in $\Uc_{n+1}$ as neighborhoods in $\Uc_n$.
Also, $\Uc_{n+1}$ includes $\Vc$.
Thus we have $s_{n+1}\le 2qs_n+
\mathrm{diam}(\wt U^*)$.  

Thus, $s_n$ satisfies Lemma \ref{l:sn} with 
$b=\mathrm{diam}(\wt U^*)$; in particular, $s_n\to 0$ as $n\to\infty$, and so
the diameter of $\Delta_n\sm U^*$ tends to $0$.
Replacing $U^*$ with a smaller neighborhood of $K^*$ and repeating the same argument
yields $\Delta_n\subset U^*$ for large $n$. Evidently, this completes the proof.
\end{proof}

If, in the setting of Lemma \ref{l:Del}, the sequence $n_{i+1}-n_i$ does not tend to infinity, then
 it takes the same value infinitely many times.
We now consider this case.

\begin{lem}
  \label{l:bnd}
  Suppose that $n_{i+1}-n_i$ takes the same value $N$ infinitely many times.
Then the quadratic argument of $\om_2$ is $\si_2^N$-fixed and the associated pullbacks
$[\al_n,\be_n]$ of $[\al, \be]$ are shrinking segments around corresponding points of the orbit of
the quadratic argument of $\om_2$.
\end{lem}

\begin{proof}
Let $R_n$ and $L_n$ be the rays with quadratic arguments $\al_n$ and $\be_n$ landing in $K^*$ and bounding $\Delta_n$ near $K^*$.
Since $\Delta_{n+1}$ is a $P$-pullback of $\Delta_n$, then $\al_n=2\al_{n+1}\pmod 1$ and $\be_n=2\be_{n+1}\pmod 1$,
and the interval $[\al_{n+1},\be_{n+1}]$ is twice shorter than $[\al_n,\be_n]$.
By the assumption, $\om_2\in\Delta_n\cap\Delta_{n+N}$ for arbitrarily large $n$.
Hence the intervals $[\al_n,\be_n]$ contain the quadratic argument of $\om_2$.
Passing to the limit, we see that this quadratic argument is $\si_2^N$-fixed. 
The last claim now follows.
\end{proof}

Standard arguments now yield the next lemma.

\begin{lem}\label{l:all-n}
If 
$\om_2$ is non-recurrent and a decoration $D$ has a quadratic argument $\al(D)$
which does not belong to the orbit of a periodic quadratic argument $\al(D_c)$,
then there exists $N$ such that any $P^N$-pullback of $D$ adjacent to $K^*$ is contained
in $U^*$.
\end{lem}

\begin{proof}
If $D_n$ is a $P^n$-pullback of $D$ adjacent to $K^*$ with $D_n\not\subset U^*$, then
$f^i(D_n)\not\subset U^*$ for every $i, 0\le i\le n$. By Lemma \ref{l:bnd} any backward pullback
orbit of $D$ adjacent to $K^*$ satisfies conditions of Lemma \ref{l:Del}; thus, for any such orbit $D=D_0,$
$D_1, $ $\dots$ there exists the minimal $n$ such that $D_n\subset U^*$. By way of contradiction suppose that for any $N$ there exists
a $P^N$-pullback of $D$ adjacent to $K^*$ and not contained in $U^*$. Then there are infinitely many such pullbacks. Now
take the $P$-pullbacks of $D$ adjacent to $K^*$ (there are finitely many of them). Choose among them one pullback that has
$P^N$-pullbacks adjacent to $K^*$ and not contained in $U^*$ for any $N$, and then apply the same construction to it. Evidently, in the end
we will construct a backward pullback orbit of $D$ adjacent to $K^*$, a contradiction with Lemma \ref{l:Del}.
\end{proof}

All this implies 
Proposition \ref{p:alldec}.

\begin{prop}
  \label{p:alldec}
If $f$ is immediately renormalizable, 
and $\om_2$ is non-recurrent then
every decoration is eventually mapped to $D_c$.
\end{prop}

\begin{proof}
Suppose that the quadratic argument $\al(D)=\ga$ of $D$ does not belong to
a periodic orbit of $\al(D_c)$. Then by Lemma \ref{l:all-n} for any $U^*$ the $P^n$-pullbacks
of $D$ will be contained in $U^*$ for any $n>N(U^*)$ where $N(U^*)$ depends on $U^*$.
Consider now the union $K_d$ of $K^*$ and all decorations that are eventually mapped to $D_c$.
We claim that the set $K_d$ is backward invariant.
Indeed, take any decoration $D\subset K_d$.
Then any $f$-pullback of $D$ is either a decoration in $K_d$ or a subset of $D_c$.
It follows that $f^{-1}(K_d)\subset K_d$, as desired.
Also, the set $K_d$ is compact as a union of $K^*$ and a 
sequence of sets
that are closed in $\C\sm K^*$ and accumulate to $K^*$. Observe that if $\al(D_c)$ is
periodic, the sets in $K_d$ are decorations with periodic arguments from the
$\si_2$-orbit of $\al(D_c)$, or decorations with non-periodic arguments-preimages
of $\al(D_c)$.

Now, $J(f)$ is a minimal by inclusion compact backward invariant subset of $\C$.
(Equivalently, the backward orbit of any point from $J(f)$ is dense in $J(f)$.)
Thus, $J(f)\subset K_d$ since $K_d$ contains points of $J(f)$.
On the other hand, consider a bounded Fatou component $\Omega$.
If $\Omega\not\subset K^*$ then
the boundary of $\Omega$ is a subset of $K_d$.
Therefore, it is a subset of some decoration $D\subset K_d$.
It follows that $\Omega$ itself is included into $D$.
We showed that $K=K_d$ which completes the proof of the proposition.
\end{proof}

The next corollary follows from Proposition \ref{p:alldec}.

\begin{cor}
  \label{c:uni-dec}
If $f$ is immediately renormalizable 
and $\om_2$ is non-re\-cur\-rent then for each quadratic angle $\ga$ 
there is
at most one decoration with quadratic argument $\ga$.
\end{cor}

\begin{proof}
Assume that $\al(D_c)$ is non-periodic. If there are two distinct decorations
$D'$ and $D''$ with the same quadratic argument $\ga$, then there exists a unique number $n$ such that
$\si_2(\ga)=\al(D_c)$. This is the only chance for $D'$ and $D''$ to be mapped to $D_c$. It follows
from Proposition \ref{p:alldec} that $f(D')=f(D'')=D_c$. However by Lemma \ref{l:img-sec} this is impossible
(for small open sectors with argument arc containing $\al(D_c)$ the map $f$ restricted on them is one-to-one
by Lemma \ref{l:img-sec}).

Now suppose that $\al(D_c)$ is periodic of period $n$. Then, similar to the previous paragraph, we see that if
there are two distinct decorations $D'$ and $D''$ with the same quadratic argument $\ga$ then there must exist
a decoration $D'''\ne D_c$ with quadratic argument $\al(D_c)$ such that $f^n(D''')=D_c$. Transferring this to the $z^2$-plane
we see that the map $z^2$ is not one-to-one in a small neighborhood of the point of the unit circle with
argument $\al(D_c)$, a contradiction completing the proof.
\end{proof}

The next lemma specifies properties of the gap $\Uf$.

\begin{lem}\label{l:perio-type}
If $\om_2$ is non-recurrent, then $\Uf$ is of periodic type.
\end{lem}

\begin{proof}
Consider the critical value decoration $D_v$.
By Proposition \ref{p:alldec}, $D_v$ eventually maps back to $D_c$.
The rays on the boundary of the minimal sector $S(\om_2)$ map to the rays on the boundary of $S(f(\om_2))$ under $f$.
This follows from Lemma \ref{l:bdminsec}.
The sector $S(f(\om_2))$ is eventually mapped to $S(\om_2)$ by Lemma \ref{l:noncr-sec}.
Therefore, $\Uf$ is of periodic type.
\end{proof}

\section{Main theorem}\label{s:mt}

To prove the main result we rely upon recent powerful results obtained in \cite{dl21}.
To state them, we need to state a property of quadratic polynomials from the Main Cardioid
of the Mandelbrot set established in \cite{chi08}.

\begin{thm}[\cite{chi08}]\label{t:chi}
Let $Q$ be a quadratic polynomial with a fixed
point $a$ such that $P'(a)=e^{2\pi i \ta}$ with irrational $\ta$. Then the limit set
$\om(c_Q)$ of the critical point $c_Q$ of $Q$ is a continuum.
\end{thm}

In the situation of Theorem \ref{t:chi} the set $\thu(\om(c_Q))$ is called the \emph{mother hedgehog} of $Q$ and is
denoted by $M_Q$.

\begin{thm}[\cite{dl21}]\label{t:md}
Let $Q$ be a quadratic polynomial with a fixed point $a$ such that $Q'(a)=e^{2\pi i \ta}$
with irrational $\ta$. Then $Q^{-1}(M_Q)$ is the union of $M_Q$ and a continuum $M_Q'$ such
that $M_Q\cap M_Q'=\{c\}$.
\end{thm}

Theorem \ref{t:md} allows one to view $Q$ similar to how Siegel quadratic
polynomials with locally connected Julia sets are viewed. By Theorem \ref{t:md} there are infinite concatenations of
pullbacks of $M_Q$ analogous to external rays in that they partition $J(Q)$ in
pieces and, thus, enable further study of topology $J(Q)$. Indeed, by Theorem \ref{t:md} countably many
pullbacks of $M_Q'$ are attached to $M_Q$ at preimages of $c_Q$ that belong to $M_Q$. Each of them is eventually mapped onto
$M_Q$. Hence each of them has countably many attached to it pullbacks of $M_Q$, etc. The entire
grand orbit of $M_Q$ can be viewed as a ``spiderweb'' of concatenated in various ways pullbacks of $M_Q$.

Each pullback $T$ of $M_Q$ can be characterized by a sequence of integers $\thr(T)$ called the \emph{thread} of $T$.
It reflects the ``journey'' of $T$ to $M_Q$,
and, simultaneously, the concatenation (the \emph{chain}) $\chai_T$ (of pullbacks of $M_Q$) that connects $M_Q$ and $T$.
In what follows we denote chains by boldface capital letters (e.g., $\bA$ or $\bX$) and threads by small boldface letters (e.g.,
$\ba$ and $\bx$). To study chains of pullbacks of $M_Q$ we need the next lemma.

\begin{lem}\label{l:chains}
Distinct pullbacks $S$ and $T$ of $M_Q$ can only intersect if they meet at a point that is a preimage of
$c_Q$. Two chains share an initial finite string and are otherwise disjoint.
\end{lem}

\begin{proof}
Observe that $M_Q$ is a full continuum. Therefore all pullbacks of $M_q$ (including $M'_Q$) are full continua too.
It follows that if $T$ is a pullback of $M_Q$ and $c_Q\notin T$, then $P|_T$ is a homeomorphism to the image.
Suppose that $S\ne T$ are pullbacks of $M_Q$ and $x\in S\cap T$.
Let us apply $Q$ to $S$ and $T$ step by step. Then there are two cases.

(1) At some earliest moment $n$ we have $Q^n(S)=M_Q$ and  $Q^n(T)\ne M_Q$. Then $Q^n(T)$ is a $Q^m$-pullback of
$M'_Q$ for some $m\ge 0$. It follows that $Q^{m+n}(x)=c_Q$ as desired.

(2) At some earliest moment $n$ we have $Q^n(S)=Q^n(T)=A$, for each $i<n,$ $Q^i(S)\ne Q^i(T)$ and neither of these
sets equals $M_Q$. Then $Q^n(S)\ne M_Q$ either because otherwise we must have $Q^{n-1}(S)=M_Q$, a contradiction.
The set $Q^{n-1}(S\cup T)=B$ is a pullback of the set $A$. Together with the fact that $Q|_B$ is not a homeomorphism this
implies that $c_Q\in Q^{n-1}(S)\cap Q^{n-1}(T)$. Since neither of these two sets equals $M_Q$ and they are distinct, it follows that
at most one of them equals $M'_Q$ while the other one (denote is by $Z$) is a pullback of $M'_Q$. We have that $c_Q\in Z\cap M'_Q$ and, for
some $k>0$, $Q^k(Z)=M'_Q$ while $Q^k(M'_Q)=M_Q$. It follows that $Q^K(c_Q)=c_Q$, a contradiction.

We leave the second claim of the lemma to the reader.
\end{proof}

Threads are one-sided sequences, infinite on the left. For each pullback $S$ in $\chai_T$
choose the least number $m(S)$ such that $Q^{m(S)}(S)=M_Q$ while $M_Q$ itself is associated with infinite on the left
string of zeros denoted by $\ol{0}$. Thus, if $\chai_T=(M_Q, S_1, \dots, S_{k-1}, S_k=T)$ then $\thr(T)=(\ol{0}, m(S_1), \dots, m(T))$ so that
$\thr(M_Q)=(\ol{0}),$ $\thr(M'_Q)=(\ol{0}, 1)$ etc.
Observe that given a pullback $T$ of $M_Q$, its thread $\thr(T)$ is a sequence with $\ol{0}$ followed by a finite
string of strictly growing positive integers.

If $\thr(T)=(\ol{0}, m_1, \dots, m_k)$ then $S_1$ is the $Q^{m_1}$-pullback of $M_Q$ attached to $M_Q$.
Under $Q^{m_1}$ the set $S_2$ maps forward so that $Q^{m_1}(S_2)$ is attached to $M_Q$ and is later, under the action of $Q^{m_2-m_1}$, mapped to $M_Q$.
In general, each power $Q^{m_j}$ moves all pullbacks $S_i\in \chai_T, j\le i\le k$ forward so that the
chain $Q^{m_j}(M_Q, S_1, \dots, S_{k-1}, S_k)$ equals
$$(M_Q, Q^{m_j}(S_{j+1}), Q^{m_j}(S_{j+2}), \dots, Q^{m_j}(S_k))$$
and the associated thread
is $(\ol{0}, m_{j+1}-m_j, \dots, m_k-m_j)$. Thus,
the action of $Q$ on threads will be denoted by $\eta$ and translates as
$\eta: (\ol{0}, m_1, \dots, m_k)\mapsto (\ol{0}, m_1-1, \dots, m_k-1)$
(notice that if $m_1=1$ the number $m_1-1=0$ should be viewed as a part of infinite
strings of zeros).

Expanding this idea, consider infinite chains (of pullbacks of $M_Q$) and characterize them by their \emph{threads},
i.e. infinite in both directions sequences of integers $(\ol{0}, m_1, m_2, \dots, )$ where $0<m_1<\dots$
(we will still use the same notation for infinite chains/threads as for finite ones). The corresponding
infinite chain of pullbacks of $M_Q$ is $(M_Q, S_1, \dots)$ that can be described as follows: the set $S_1$ is the $Q^{m_1}$-pullback of $M_Q$,
attached to $M_Q$, the set $S_2$ is the $Q^{m_2}$-pullback of $M_Q$, attached to $S_1$, etc. 
Clearly, the map $\eta$ can be defined on infinite threads as follows: if $\bx=(\ol{0}, m_1, m_2, \dots)$, then $\eta(\bx)=(\ol{0}, m_1-1, m_2-2, \dots)$.
This action corresponds to the action of $Q$ on infinite chains so that if $\thr(\bX)=\bx$, then $\thr(Q(\bX))=\eta(\bx))$
(i.e., $\thr\circ Q=\eta\circ \thr$). Clearly, the map $\eta$ defined on infinite threads has periodic points.
E.g., by definition $\ba=(\ol{0}, m_1, m_2, \dots)$ is $\eta$-fixed if $m_1-1=0,$
$m_2-1=m_1,$ etc, i.e. the only $\eta$-fixed thread is $(\ol{0}, 1, 2, 3, \dots)$. Of course, there are other periodic infinite chains.

Suppose that $(\ol{0}, m_1, \dots)$ is periodic of (minimal) period $N$ (in what follows by ``period'' we always mean ``minimal period'').
It is easy to see that this means the following: there exists $k$ such that
$m_k=N$ and, moreover, for any $j=lk+r, 0\le l, 0\le r<k$ we have $m_j=lN+m_r$.
Recall that the \emph{topological hull} $\thu(A)$ of a compact set $A\subset \C$ is the complement
of the unbounded complementary domain of $A$.

\begin{lem}\label{l:exponent}
Let $Q$ be a quadratic polynomial with a fixed point $a$ such that $Q'(a)=e^{2\pi i \ta}$
with irrational $\ta$. Then claims $(1) - (4)$ hold.

\begin{enumerate}

\item Let $T$ be a pullback of $M_Q$ disjoint from $M_Q$. Then there exists $C_1>0$ and $0<q<1$ such that
any $Q^N$-pullback of $T$ has diameter less than $C_1q^N$ for any $N$.

\item Distinct infinite chains 
converge to distinct points.

\item Suppose that $\thr(\bX)=\bx$ is periodic of period $N$. Then $\bX$ converges to a periodic point of period $N$
that does not belong to $M_Q$. Any preperiodic chain converges to a preperiodic point.

\item The only periodic point in $M_Q$ is the fixed point $a$. All periodic points $y\ne a$ of $Q$ are limits of the
corresponding chains of pullbacks of $M_Q$.

\end{enumerate}

\end{lem}

\begin{proof}
(1) The fact that $M_Q=\thu(\om(c_Q))$ implies that $T$ is disjoint from $\om(c_Q)$. By Ma\~n\'e \cite{man93}
(see Theorem \ref{t:mane}) we can cover $\thu(T)$ with small Ma\~n\'e disks $U_1, \dots, U_k$ so that their
union is itself a Jordan disk $V$ with $\ol{V}\cap M_Q=\0$ so that any $Q^n$-pullback of each $U_i$ is of diameter less than $Cq^n$
for some $C>0$ and $0<q<1$ (these depend solely on $Q$). Since $\ol{V}\cap \thu(\om(c_Q))=\0$ then any
$Q^N$-pullback of $V$ is a topological disk $V'$ that homeomorphically maps onto $V$ under $Q^N$
(observe that since $c_Q$ is recurrent, the orbit of $c_Q$ is contained in $\om(c_Q)$. Hence each $U_i$ is
represented in $V'$ by exactly one pullback $U'_i$ whose diameter, by Theorem \ref{t:mane}, is less than $Cq^N$.
Hence the diameter of $V'$ is less than $kCq^N$ and it remains to set $C_1=kC$.

(2) Now, if two distinct
chains of pullbacks of $M_Q$ converge to the same point, it would follow that
some points of these pullbacks are blocked from infinity by the union of these chains and, therefore, cannot belong to $J(Q)$,
a contradiction. Observe that we are not claiming that all chains converge. However it two chains do converge,they cannot converge
to the same as was just proven.

(3) If $\bx=(\ol{0}, m_1, \dots)$, then there exists $k$ such that
$m_k=N$ and, moreover, for any $j=lk+r, 0\le l, 0\le r<k$ we have $m_j=lN+m_r$.
Suppose that $S$ is the $j$-th element of $\bX$ (associated with $m_j$). It is
preceded in $\bX$ by another pullback $S'$ of $M_Q$, associated with $m_{j-1}$,
and until $S'$ maps to $M_Q$ under $Q^{m_j}$, the set $S$ remains detached from $M_Q$.
Thus, $Q^{m_{j-1}-1}(S)$ is a specific pullback of $M_Q$ detached from $M_Q$. Evidently, there is a finite
collection of such ``last detached'' pullbacks of $M_Q$ that serve all elements of $\bX$. By (1) we can choose
numbers $C>0$ and $0<q<1$ that serve all these ``last detached'' pullbacks, and, after straightforward adjustments,
conclude that $\mathrm{diam}(S)<Cq^{m_{j-1}}$. Evidently, this implies that $\bX$ converges to a $Q^N$-fixed point $z$.
If the period of $z$ were less than $N$, there would exist a chain distinct from $\bX$ but converging to $z$,
a contradiction to (2). Moreover, $z$ cannot belong to $M_Q$ as otherwise some points of $M_Q$ would be blocked from
infinity (similar to the argument in (2)). The last part of claim (3) now follows.

(4) Claims (1) - (3) hold for some quadratic polynomial $P_{sie}$ with Siegel fixed point $b$ and
locally connected Julia set. In case of such polynomials it is easy to check that all periodic points
of $P_{sie}$ except for $b$ can be obtained as limits of periodic chains. However chains of $Q$ and chains
of $P$ are in one-to-one correspondence with corresponding periodic threads.
Hence the number of periodic points of any period $n$ is the same for $Q$ and for $P_{sie}$
and coincides with the number of periodic threads of that period.
Recall that except for $b$ there are no $Q$-periodic points that belong to $M_Q$ as $M_Q$ is a $Q$-invariant Jordan curve
on which $Q$ is conjugate with an irrational rotation. It follows from these counts of periodic points for $Q$ and $P_{sie}$ that
no periodic point of $Q$, except for $a$, can belong to $M_Q$.
\end{proof}

Let us now combine Lemma \ref{l:exponent} and tools from \cite{bfmot12}.

\begin{thm}\label{t:allc}
Let $P$ be a polynomial of any degree with a fixed non-repelling point $a\in B^*$
where $B^*$ is an invariant filled quadratic-like
Julia set of $P$. Assume that any periodic neutral point $x\ne a$ of $P$ is parabolic.
Then for a periodic point $z\in B^*\sm \{a\}$ and a periodic external ray
$R$ of $P$ the fact that $z$ belongs to the impression $I$ of $R$ implies that
I=$\{z\}$ (thus, $J(P)$ is locally connected at $z$).
\end{thm}

It is well-known that if the impression of a ray is degenerate then this rays lands
at some point $z$ and the continuum is ready.


\begin{proof}
First assume that $a$ is either attracting or parabolic. Then $P$ has no Cremer or Siegel points.
Hence by Corollary 7.5.4 \cite{bfmot12} any periodic impression is degenerate.

The forthcoming argument which involves
chains, thread etc applies to $P|_{B^*}$ as $B^*$ is quadratic-like; we will, therefore, use the same notation as before
despite the fact that $P$ itself is not quadratic.
Suppose now that $P'(a)=e^{2\pi i \ta}$ where $\ta$ is irrational.
We will need the following construction. Let $\bX$ be the chain that converges to $z$.
Let $(\ol{0}, m_1, \dots)=\thr(\bX)=\bx$ be its thread. Choose a pullback $S$ of the mother hedgehog $M_P$ of $P$ that belongs to
$\bX$ assuming that $S\ne M_P,$ $S|ne M'_P$.. Evidently, on the plane there are pullbacks $L, R$ of $M_P$ attached to $S$ from either side of the part
of $\bX$ connecting $S$ and $z$.
The choice of
$L$ and $R$ depends on  $\ta$, however, regardless of $\ta$, such $L$ and $R$ exist.
Then choose (pre)periodic chains
$\chai_l$ and $\chai_r$ that extend $L$ and $R$ and converge to points $y_l, y_r$, and external rays
$Y_l, Y_r$ of $P$ that land at $y_l, y_r$ and have arguments $\ta_l, \ta_r$, resp. Set
$$Z=Y_l\cup \{y_l\} \cup L\cup S\cup R\cup \{y_r\}\cup Y_r$$
and use it in the proof. Repeat this construction for a different pullback $S'\in \bX$ of $M_P$
assuming that in $\bX$ the set $S'$ is closer to $M_P$ than $S$ and construct a similar set
$Z'$ for $S'$ so that $Z$ separates $Z'$ from $z$.

The set $Z$ divides $\C$ in two open subsets, $W_a\ni a$ and $W_z\ni z$. A ray whose impression contains $z$ must be
contained in $W_z$ or coincide with $Y_l$ or with $Y_r$. Now, the set $Z'$ divides $\C$ in two open subsets, $W'_a\ni a$ and $W'_z\ni z$,
and $W_z'\supset Z\cup W_z$. It follows that the impression of any angle that contains $z$ cannot contain $a$. By Corollary 7.5.4 from \cite{bfmot12}
this impression is degenerate and, hence, coincides with $\{z\}$. This implies that $J(P)$ is locally connected at $z$ as desired.
\end{proof}

The next result is used in dealing with an easier case.

\begin{thm}[\cite{bot21}]\label{t:cutpts}
Let $P$ be a polynomial of any degree with connected filled Julia set $K(P)$.
Let $x\in K(P)$ be a repelling or parabolic periodic point
and all $K(P)$-rays to $x$ form $m$ wedges $W_i$, where $1\le i\le m$.
Moreover, suppose that $x\in Q$ is a cutpoint of order $n$ of an invariant continuum $Q\subset J$.
Then $n\le m$, each wedge $W_i$ intersects $Q$ over a connected (possibly empty) set,
 and every $Q$-ray to $x$ is isotopic rel. $Q$ to a $K(P)$-ray that lands at $x$.
\end{thm}

We are finally ready to prove Theorem \ref{t:recur}.

\begin{proof}[Proof of Theorem \ref{t:recur}]
Since $\om_2$ is non-recurrent, then by Lemma \ref{l:perio-type} $\Uf$ is of periodic type.
Let $M=\ol{\al\be}$ be its major and $I=(\al, \be)$ be its major hole. Then $\al$ and $\be$ are
$\si_3$-periodic angles each of which is approached from the outside of $I$ by periodic angles
$\al_i\to \al$ and $\be_i\to \be$ whose rays land in $J^*$. These angles
belong to $\Uf$ and correspond (through the map $\tau$ collapsing edges of $\Uf$) to $\si_2$-periodic angles $\al'_i$ and $\be'_i$.
Moreover, ``quadratic'' angles $\al'_i$ and $\be'_j$ converge to the same ``quadratic'' periodic angle $\ga=\tau(M)$.
Then by Theorem \ref{t:allc} the landing points of ``quadratic'' external rays with arguments $\al'_i$ and $\be'_i$
converge to the periodic landing point of the ``quadratic'' ray with argument $\gamma$. Evidently, the same will
happen in $K^*$, i.e. the landing points of $P$-rays with arguments $\al_i$ and $\be_j$ converge to a periodic point in $K^*$.
This point belongs to the impressions of $\al$ and of $\be$.  Hence, by Theorem \ref{t:allc}, they are degenerate and coincide with this point, a contradiction
with the assumption that the rational lamination of $P$ is empty.
\end{proof}

\end{document}